\documentclass[10pt,a4paper]{amsart}

\usepackage[utf8]{inputenc}
\usepackage[british]{babel}
\usepackage{amsmath,amsfonts,amssymb,amsthm,mathtools}
\usepackage[backref,ocgcolorlinks,linkcolor=blue]{hyperref}
\usepackage[nameinlink]{cleveref}
\usepackage{xcolor}
\usepackage{enumitem}

\usepackage[framemethod=TikZ]{mdframed}

\mdfdefinestyle{MyFrame}{%
    linecolor=black,
    outerlinewidth=1.5pt,
    roundcorner=20pt,
    innertopmargin=5pt,
    innerbottommargin=\baselineskip,
    innerrightmargin=20pt,
    innerleftmargin=20pt,
    backgroundcolor=gray!10!white}

\newtheorem*{main-theorem}{Theorem~\ref{final}}
\newtheorem*{main-lemma}{Theorem~\ref{bound}}
\newtheorem{proposition}{Proposition}[section]
\newtheorem{corollary}[proposition]{Corollary}
\newtheorem{lemma}[proposition]{Lemma}
\newtheorem{theorem}[proposition]{Theorem}

\theoremstyle{remark}
\newtheorem{remark}[proposition]{Remark}
\newtheorem*{acknowledgements}{Acknowledgements}

\newcommand{\R}{\mathbb{R}}
\newcommand{\C}{\mathbb{C}}

\newcommand{\dd}{\mathrm{d}}
\renewcommand{\Re}{\mathrm{Re}\,}
\renewcommand{\Im}{\mathrm{Im}\,}
\renewcommand{\le}{\leqslant}
\renewcommand{\ge}{\geqslant}
\renewcommand{\leq}{\leqslant}
\renewcommand{\geq}{\geqslant}

\DeclareMathOperator*{\supp}{supp} % support
\author{Pedro Caro}
\address{BCAM - Basque Center for Applied Mathematics}
\email{pcaro@bcamath.org}

\author{Mar\'ia \'Angeles Garc\'ia-Ferrero}
\address{Universitat de Barcelona}
\email{garciaferrero@ub.edu}

\author{Keith M. Rogers}
\address{Instituto de Ciencias Matem\'aticas CSIC-UAM-UC3M-UCM}
\email{keith.rogers@icmat.es}

\date{\today}
\thanks{Partially supported by Ikerbasque and BERC  2022-2025, the MICINN grants PID2021-125021NAI00, PID2021-122154NB-I00, PID2021-122156NB-I00 and PID2021-124195NB-C33,  the Severo Ochoa grants CEX2019-000904-S and CEX2021-001142-S, and the ERC grant AdG-834728.
}

\title[Reconstruction with Lipschitz conductivities]{Reconstruction for the Calder\'on problem with Lipschitz conductivities}

\begin{document}

\begin{abstract} 
We determine the conductivity of the interior of a body using electrical measurements on its surface. We assume only that the conductivity is bounded below by a positive constant and that the conductivity and surface are Lipschitz continuous. To determine the conductivity we first solve an associated integral equation locally, finding solutions in $H^1(B)$, where $B$ is a ball that properly contains the body. A key ingredient is to equip this Sobolev space with an equivalent norm which depends on two auxiliary parameters that can be chosen to yield a contraction.
\end{abstract}

\maketitle

\vspace{1.5cm}

%\tableofcontents
%%%%%%%%%%%%%%%%%%%% New section %%%%%%%%%%%%%%%%%%%%
\section{Introduction}

We consider the conductivity equation in a bounded domain $\Omega\subset \mathbb{R}^n$ and place electric potentials $\phi\in H^{1/2}(\partial\Omega)$ on the Lipschitz boundary $\partial\Omega$;
\begin{equation}
	\left\{
		\begin{aligned}
		\nabla \cdot (\sigma \nabla u)  &= 0&   \text{in}&\ \Omega , \\
		u|_{\partial\Omega}&= \phi.&  
		\end{aligned}
	\right.
	\label{pb:BVP}
\end{equation}
Throughout the article, the conductivity $\sigma$ is assumed to be bounded above and below by positive constants so that
\eqref{pb:BVP} has a unique weak solution $u$ in the $L^2$-Sobolev space $H^1(\Omega)$. The Dirichlet-to-Neumann map~$\Lambda_\sigma$ can then be formally defined  by
\begin{equation}\label{dn} \Lambda_\sigma :  \phi \mapsto \sigma \partial_\nu u|_{\partial \Omega},\end{equation}
where $\nu$ denotes the outward unit 
normal vector to $\partial \Omega$. This provides us with the steady-state perpendicular currents induced by the electric potentials~$\phi$.  

Motivated by the possibility of creating an image of the interior of a body from these noninvasive voltage-to-current measurements on its surface, Calder\'on asked~\cite{zbMATH05684831}  whether the conductivity $\sigma$ is uniquely determined by $\Lambda_{\sigma}$ and, if so,  whether  $\sigma$ can be calculated from~$\Lambda_{\sigma}$.  
In two dimensions, Astala and P\"aiv\"arinta answered the uniqueness part in \cite{zbMATH05050053}, as well as providing a reconstruction algorithm in \cite{zbMATH05077070}. The two-dimensional problem has distinct mathematical characteristics, so from now on we consider only $n\ge 3$.

With $n\ge 3$, it has so far been necessary  to make additional  regularity assumptions. In the eighties, Kohn and Vogelius \cite{zbMATH03939884} proved uniqueness for real-analytic conductivities and Sylvester and Uhlmann \cite{zbMATH04015323} improved this to smooth conductivities.  Nachman, Sylvester and Uhlmann \cite{zbMATH04050176} then proved uniqueness for twice continuously differentiable conductivities and Nachman~\cite{zbMATH04105476} and Novikov~\cite{zbMATH04129351} provided reconstruction algorithms. These pioneering articles provoked a great deal of interesting work, including that of Brown \cite{zbMATH00912089}, P\"aiv\"arinta--Panchenko--Uhlmann~\cite{zbMATH02005267} and Brown--Torres~\cite{zbMATH02102106} for conductivities with $3/2$ derivatives. In the past decade, a breakthrough was made by Haberman and Tataru \cite{zbMATH06145493}, who proved uniqueness for continuously differentiable conductivities or Lipschitz conductivities with $\|\nabla \log\sigma\|_\infty$ sufficiently small. Garc\'ia and Zhang~\cite{zbMATH06659335} then provided a reconstruction algorithm under the same assumptions. Two of the authors~\cite{zbMATH06534426} removed the smallness condition from the uniqueness result and the purpose of this article will be to extend this work to a reconstruction algorithm that holds for all Lipschitz conductivities. We will not assume that the conductivity is constant near the boundary, nor will we extend the conductivity in order to achieve this, 
% Moreover,  we will not need to calculate $\nu\cdot\nabla \sigma|_{\partial\Omega}$,  
 leading to simpler formulas than those of~\cite{zbMATH06659335}; see Section~\ref{recon}.

Before we outline the proof, we remark that there are also uniqueness results for conductivities in Sobolev spaces; see \cite{zbMATH06490961,  zbMATH07373390, zbMATH07395052}. In particular, Haberman~\cite{zbMATH06490961} proved that uniqueness holds for bounded conductivities in $W^{1,n}(\overline{\Omega})$, with $n=3$ or~$4$. Note that this is a strictly larger class than Lipschitz, however there are obstacles to reconstruction via their methods; see Remark~\ref{dif} for more details. It has been conjectured that Lipschitz continuity is the sharp threshold within the scale of H\"older continuity; see for example \cite{zbMATH00912089} or \cite[Open Problem 1]{UhlmannICM}.

When $\sigma$ is Lipschitz,  weak solutions to~\eqref{pb:BVP} are in fact strong solutions; see for example \cite[Theorem 1.3]{zbMATH06012150}. 
 Defining the Dirichlet-to-Neumann map as in \eqref{dn} by identifying $\sigma \partial_\nu u|_{\partial \Omega}$ with the normal trace of   $\sigma \nabla u$, we have the divergence identity
\begin{equation*}\label{div}
\int_{\partial\Omega} \Lambda_\sigma[\phi] \psi =\int_{\Omega} \sigma \nabla u\cdot \nabla \psi
\end{equation*}
whenever $(\phi,\psi)\in H^{1/2}(\partial\Omega)\times H^{1}(\Omega)$; see for example \cite[Proposition 2.4]{zbMATH07578602}. Given this identity, it is possible to describe the  heuristic which underlies the reconstruction: For each $\xi\in \mathbb{R}^n$, one hopes to choose an oscillating pair $(\phi,\psi)$  so that the right-hand side becomes a nonlinear Fourier transform of $\sigma$ evaluated at~$\xi$. As the left-hand side can be calculated from the measurements, the conductivity might then be recoverable by Fourier inversion. Indeed, much of the literature, including the original work of Calder\'on~\cite{zbMATH05684831}, has involved pairs $(e^{\rho \cdot {\rm x}},e^{\rho' \cdot {\rm x}})$ with  $\rho,\rho'\in \mathbb{C}^n$ chosen carefully so that $\rho+\rho'$ is equal to a real constant multiple of $-i\xi$, where $i:=\sqrt{-1}$. The hope is that  the essentially harmonic $u$ is not so different from $e_\rho:=e^{\rho \cdot {\rm x}}$, and so the complex vector~$\rho$ is chosen in such a way that  $\rho\cdot\rho=0$ so that $e_\rho$ is  harmonic.

In fact we begin by noting that $u$ is a solution to the conductivity equation if and only if $v=\sigma^{1/2} u$ is a solution to the Schr\"odinger equation
\begin{equation}\label{schrod}
\Delta v= q v {\quad}  \text{in}\ \Omega,
\end{equation}
where formally $q = \sigma^{-1/2} \Delta \sigma^{1/2}$. Kohn and Vogelius \cite{zbMATH03957806} observed that if~$\sigma|_{\partial\Omega}$ and $\nu \cdot \nabla \sigma|_{\partial\Omega}$ are known, then the Dirichlet-to-Neumann map~$\Lambda_q$ for the Schr\"odinger equation  \eqref{schrod} can be written in terms of~$\Lambda_\gamma$, and so the literature has mainly considered the essentially equivalent problem of recovering $q$ from~$\Lambda_q$ (which is intimately connected to inverse scattering at fixed energy). We will only partially use the equivalence however: we will recover $q$ directly from~$\Lambda_\gamma$, circumventing the need to calculate $\nu \cdot \nabla \sigma|_{\partial\Omega}$. This is connected to the fact that our  conductivities are not regular enough to define~$q$ in a pointwise fashion. However, as noted by Brown~\cite{zbMATH00912089}, 
it suffices to define $\big\langle qv, \psi \big\rangle:=\big\langle q, v\psi \big\rangle$ for suitable test functions~$\psi$, with
\begin{align}\label{qog}
\big\langle q,\, \centerdot\, \big\rangle &:= - \int_{\Omega} \nabla \sigma^{1/2} \cdot \nabla (\sigma^{-1/2}\,\centerdot\,). \end{align}
By the product rule and the Cauchy--Schwarz inequality, $\langle q, \, \centerdot\, \rangle$ and $\langle qv, \, \centerdot\, \rangle$ are bounded linear functionals on $H^1(B)$, where $B$ is a ball that properly contains~$\Omega$, so in particular we can make sense of $q$ and~$qv$ as distributions.

Rather than solving \eqref{schrod} directly, we consider solutions to the Lippmann--Schwinger-type equation 
\begin{equation}\label{ls0} v=\Delta^{\!-1}\circ {\rm M}_q[v]+e_{\rho},
\end{equation}
where ${\rm M}_q:f\mapsto qf$ and the inverse of the Laplacian is defined using the Faddeev fundamental solution; see Section~\ref{fadd}.
Integral equations like this are usually solved globally, however we will solve this only locally, finding solutions $v\in H^1(B)$. 
Writing $
v=e_{\rho}(1 + w)$
and additionally requiring that the remainders~$w$ vanish in some sense as~$|\rho|\to \infty$ gives hope that the nonlinear Fourier transform will converge to the linear Fourier transform in the limit. Solutions of this type were introduced to the problem by Sylvester and Uhlmann~\cite{zbMATH04015323} and have since become known as CGO solutions, where CGO stands for Complex Geometrical Optics. Substituting  into~\eqref{schrod} and multiplying  by~$e_{-\rho}$, we find that 
\begin{equation}\label{rem}
\Delta_{\rho} w= {\rm M}_q[1+w] {\quad}  \text{in}\ \Omega,
\end{equation}
where $\Delta_\rho:=\Delta+2\rho\cdot\nabla$. In much of the  literature $\Delta_\rho$ is inverted using the Fourier transform and the resulting integral equation 
is solved globally via a contraction for $\Delta^{\!-1}_\rho\circ {\rm M}_q$ and Neumann series. In order to reconstruct~$\sigma$ from~$\Lambda_\sigma$ (as opposed to just proving uniqueness),  we must additionally determine which electric potentials should be placed on the boundary in order to generate the CGO solutions. A contraction for $\Delta^{\!-1}_\rho\circ {\rm M}_q$ can also be helpful in this step, however, the need for such a contraction was circumvented in the uniqueness result of~\cite{zbMATH06534426}, instead  solving the differential equation~\eqref{rem} via the method of {\it a priori} estimates.  

Nachman and Street were able to recover the boundary values of CGO solutions that had been constructed via {\it a priori} estimates~\cite{zbMATH05690567}, however we were unable to take advantage of their ideas; see Remark~\ref{street} for more details. Instead we will reprove the existence of CGO solutions, this time via Neumann series, however we will adopt the previously mentioned intermediate approach of solving the integral equation only locally. That is to say, we find a $w\in H^1(B)$ such that
\begin{equation}\label{int}
\big(\mathrm{I}-\Delta^{\!-1}_\rho\circ {\rm M}_q\big)w=\Delta^{\!-1}_\rho\circ {\rm M}_q[1],
\end{equation}
where the identity holds as elements of $H^1(B)$. This is equivalent to \eqref{ls0} when $\Delta^{\!-1}\circ {\rm M}_q$ is defined appropriately; see Remark~\ref{9.2}.

Most of the article will be occupied by the proof of the contraction for  $\Delta^{\!-1}_\rho\circ {\rm M}_q$ in Sections~\ref{sketch}-\ref{contraction}. In Section~\ref{sketch} we give a sketch of its proof before proving the key Carleman estimate in Section~\ref{sec:Tzeta}. In Section~\ref{spaces} we incorporate the associated convex weights into our localised versions of the Haberman--Tataru norms,  so that they not only depend on~$\rho$ but also on an auxiliary parameter $\lambda>1$. The final estimate for $\Delta^{\!-1}_\rho$, proved in Section~\ref{trup}, is somewhat weaker and easier to prove than the main estimate of~\cite{zbMATH06534426}, so the present article also simplifies the uniqueness result of~\cite{zbMATH06534426}. In  Section~\ref{third} we bound~${\rm M}_q$ with respect to the new norms and in Section~\ref{contraction} we choose the parameters in order to yield the contraction. 

In the following Section~\ref{fad} we list some of the main definitions before presenting the reconstruction algorithm in Section~\ref{recon}. The reconstruction formulas will not make mention of the new norms which are only used  in Section~\ref{finalsection} to prove the validity of the formulas. In the final Section~\ref{pract} we suggest some simplifications that could make the algorithm easier to implement.

\begin{acknowledgements} The third author thanks Kari Astala, Daniel Faraco, Peter Rogers and Jorge Tejero for helpful conversations.
\end{acknowledgements}

 %%%%%%%%%%%%%%%%%%%% New section %%%%%%%%%%%%%%%%%%%%
 \section{Preliminary notation}\label{fad}

 We invert our main operator~$\Delta_{\rho}$ initially on the space of Schwartz functions~$\mathcal{S}(\R^n)$, using the Fourier transform defined by 
$$\widehat{f}(\xi) :=  \int_{\R^n} e^{-i \xi \cdot x} f(x) \, \dd x$$
for all $\xi\in \mathbb{R}^n$  and  $f\in \mathcal{S}(\R^n)$.
 By integration by parts, one can calculate that
\begin{equation}
\label{id:symbols}
\widehat{\Delta_\rho f}(\xi)=m_\rho(\xi)\widehat{f}(\xi),\qquad \text{where}\quad m_\rho (\xi) := -|\xi|^2 + 2i \rho \cdot \xi,
\end{equation}
for all $\xi\in \mathbb{R}^n$. The reciprocal of this Fourier multiplier is integrable on compact sets, so we can define an inverse by 
\begin{equation*}\label{deft}
\Delta^{\!-1}_{\rho} g (x) := \frac{1}{(2 \pi)^{n}} \int_{\R^n} e^{i x \cdot \xi} \frac{1}{m_\rho (\xi)} \widehat{g}(\xi) \, \dd \xi
\end{equation*}
for all $x\in \mathbb{R}^n$ and $g\in \mathcal{S}(\R^n)$.

\subsection{The Faddeev fundamental solutions:}\label{fadd}   Writing the inverse Fourier transform of the product  as a convolution,  we find
\begin{equation}\label{loveit}
\Delta^{\!-1}_{\rho} g (x) = \int_{\R^n} F_{\!\rho}(x-y) g(y) \, \dd y\end{equation}
for all $x\in \mathbb{R}^n$ and $g\in \mathcal{S}(\R^n)$, where the fundamental solution $F_{\!\rho}$ for $\Delta_\rho$ is defined by
$$
F_{\!\rho}(x)
:=\lim_{r \to \infty} \frac{1}{(2 \pi)^n} \int_{\R^n}  e^{i x \cdot \xi}\frac{1}{m_\rho (\xi)} \widehat{\chi}(\xi/r) \, \dd \xi.
$$ 
Here $\chi\in  \mathcal{S}(\R^n)$ must be positive and satisfy $\widehat{\chi}(0)=1$, but the limit is insensitive to the precise choice of $\chi$ and so the integral is often  written formally, taking $\widehat{\chi}=1$. 
This fundamental solution was first considered by Faddeev in the context of quantum inverse scattering~\cite{zbMATH03237162}.

We also consider the associated fundamental solution $G_{\!\rho}:=e_{\rho}F_{\!\rho}$ for the Laplacian  and we will often write $
G_{\!\rho}({\rm x},{\rm y}):=G_{\!\rho}({\rm x}-{\rm y})$. This is not so different from the usual potential-theoretic fundamental solution. Indeed, by subtracting one from the other, one obtains a harmonic function which is thus smooth,
by Weyl's lemma; \begin{equation}\label{bigpot}
H_{\!\rho}({\rm x}):=G_{\!\rho}({\rm x})-\frac{c_n}{(2-n)}\frac{1}{|{\rm x}|^{n-2}},
\end{equation}
where $c_n$ denotes the reciprocal of the measure of the unit sphere.  For more details regarding the properties of Faddeev's fundamental solutions, see \cite[Section~6.1]{zbMATH00194126}.

\subsection{The boundary integral:} For notational compactness we  write the reconstruction formulas in terms of the bilinear functional $BI_{\Lambda_\sigma}:H^{1/2}(\partial \Omega)\times H^{1}(\Omega) \to \mathbb{C}$ defined by
\begin{equation}\label{dn2} 
BI_{\Lambda_\sigma}(\phi,\psi):=\int_{\partial\Omega} \Big(\sigma^{-1/2}\Lambda_\sigma[\sigma^{-1/2}\phi]-\nu\cdot \nabla P_0[\phi]\Big) \psi
\end{equation}
where $P_0[\phi]$ denotes the harmonic extension of $\phi$. 
 Brown~\cite{zbMATH01731190} 
 %(see also~\cite{zbMATH03939884}) 
 calculated $\sigma|_{\partial \Omega}$  from~$\Lambda_\sigma$, 
so  the boundary integral $BI_{\Lambda_\sigma}$ can be recovered from~$\Lambda_\sigma$. 
%As mentioned in the introduction,  we will not need to calculate $\nu \cdot\nabla \sigma|_{\partial\Omega}$ in contrast with previous reconstruction algorithms.
 In Lemma~\ref{ale} we will prove  that $$BI_{\Lambda_\sigma}\big(\phi,G_{\!\rho}({\rm x},\,\centerdot\,)\big) \in H^{1}(B\setminus \overline{\Omega}),$$ 
where $B$ properly contains $\Omega$. 
This allows us to define $\Gamma_{\!\!\Lambda_\sigma}:H^{1/2}(\partial \Omega)\to H^{1/2}(\partial \Omega)$ by taking the outer trace on $\partial\Omega$; \begin{equation}\label{gam}
\Gamma_{\!\!\Lambda_\sigma}[\phi]:=BI_{\Lambda_\sigma}\big(\phi,G_{\!\rho}({\rm x},\,\centerdot\,)\big)|_{\partial\Omega}.
\end{equation}
As $H_{\!\rho}$ is smooth, the singularity of $G_{\!\rho}$ is the same as that of the usual potential-theoretic fundamental solution, so $\Gamma_{\!\!\Lambda_\sigma}$ shares many properties with the single layer potential; see for example \cite[Propositions~3.8 and 7.9]{zbMATH01286366}. However,  we will not need these type of estimates going forward.

%%%%%%%%%%%%%%%%%%%% New section %%%%%%%%%%%%%%%%%%%%
\section{The reconstruction algorithm}\label{recon}

Recall our {\it a priori} assumptions, that the boundary and conductivity are Lipschitz continuous and that the conductivity is bounded below by a positive constant.

The first step of the reconstruction algorithm is to determine the electric potentials that we place on the boundary in order to generate the CGO solutions. As in the previous reconstruction formulas of \cite{zbMATH04105476, zbMATH04129351, zbMATH06659335}, we resort to the Fredholm alternative, however, once we have obtained the contraction, the argument will be direct, avoiding the use of generalised double layer potentials. The proof is postponed until Section~\ref{finalsection}.

\begin{theorem}\label{bound}\sl Consider $\rho \in \C^n$ such that  $\rho\cdot\rho=0$ and $|\rho|^2=\rho\cdot \overline{\rho}$ is sufficiently large. Let  $\Gamma_{\!\!\Lambda_\sigma}$ be defined by \eqref{gam}. Then 
\vspace{0.3em}\begin{itemize}
\item [{(i)}] $\Gamma_{\!\!\Lambda_\sigma}:H^{1/2}(\partial \Omega)\to H^{1/2}(\partial \Omega)$ is bounded compactly,
\item [{(ii)}] if $\Gamma_{\!\!\Lambda_\sigma}[\phi]=\phi$, then $ \phi=0$,
\item [{(iii)}] $\mathrm{I}-\Gamma_{\!\!\Lambda_\sigma}$ has a bounded inverse on $H^{1/2}(\partial \Omega)$,
\end{itemize}
\vspace{0.2em}
and if $v=e_{\rho}(1+w)$, where $w\in H^1(B)$ is a solution to 
 \eqref{int},  then
\vspace{0.2em}\begin{itemize}
\item [{(iv)}]  
 $v|_{\partial \Omega}= (\mathrm{I}-\Gamma_{\!\!\Lambda_\sigma})^{-1}[e_{\rho}|_{\partial \Omega}].
$
\end{itemize}
\end{theorem}

Next we provide a formula for the Fourier transform $
\widehat{q}({\rm \xi}):=\langle q, e^{-i {\rm \xi}\cdot {\rm x}}\rangle
$ where $q$ is defined in \eqref{qog}. Again we postpone the proof until the penultimate section.

\begin{theorem}\label{final}  \sl
Let $ \Pi$ be a two-dimensional linear subspace orthogonal to $ \xi \in \R^n$ and define
$$
S^1:= \Pi \cap \big\{\,\theta \in \R^n : |\theta| = 1\, \big\}.$$
 For  $ \theta \in S^1$,
 let $ \vartheta \in S^1$ be such that $ \{\theta,\vartheta \} $ is an orthonormal basis of~$ \Pi $ and define
 \begin{align*}
\rho := \tau \theta + i\Big(-\frac{\xi}{2} + \Big( \tau^2 - \frac{|\xi|^2}{4} \Big)^{1/2} \vartheta \Big),\quad \rho' := -\tau \theta + i\Big(-\frac{\xi}{2} - \Big( \tau^2 - \frac{|\xi|^2}{4} \Big)^{1/2} \vartheta \Big),
\end{align*}
where $\tau>1$. Let $BI_{\Lambda_\sigma}$  and  $\Gamma_{\!\!\Lambda_\sigma}$ be defined by \eqref{dn2} and \eqref{gam}, respectively. Then
\begin{align*}\label{qform}
\widehat{q}(\xi)=
 \lim_{T\to \infty}\frac{1}{2\pi T} \int_{T}^{2T}\!\! \int_{S^1} BI_{\Lambda_\sigma}\Big((\mathrm{I}-\Gamma_{\!\!\Lambda_\sigma})^{-1}[e_{\rho}|_{\partial \Omega}], e_{\rho'}\Big)   \,  \dd \theta \dd \tau.
\end{align*}
\end{theorem}

Finally, we recover $\sigma$ from $q$ using the approach of Garc\'ia and Zhang \cite{zbMATH06659335}.  By the work of Brown~\cite{zbMATH01731190}  and Plancherel's identity, we can now calculate  the right-hand side of
\begin{equation}\label{cond}
	\left\{
		\begin{aligned}
		\Delta w+|\nabla w|^2&=q\quad  \text{in}\ \Omega , \\
		w|_{\partial\Omega} &= \tfrac{1}{2}\log\sigma|_{\partial\Omega}.
		\end{aligned}
	\right.
\end{equation}
If $w\in H^1(\Omega)$ is the unique bounded solution  to \eqref{cond}, we then   have $$\sigma=e^{2w}\quad  \text{in}\ \Omega.$$  This completes the reconstruction algorithm. 

That  $w=\log \sigma^{1/2}$ solves~\eqref{cond} follows directly by inspection of the definition~\eqref{qog} of~$q$. For uniqueness, note that if $\tilde w$ also solved \eqref{cond}, then $u=w-\tilde w$ would solve  
\begin{equation*}
	\left\{
		\begin{aligned}
		\nabla \cdot (\gamma \nabla u)&=0\quad  \text{in}\ \Omega , \\
		u|_{\partial\Omega} &= 0,
		\end{aligned}
	\right.
\end{equation*}
where $\gamma:=e^{w+\tilde w}$. Then $u=0$ by uniqueness of solutions for elliptic equations; see for example \cite[Corollary 8.2]{zbMATH01554166}. 
%\footnote{{Recalling that by the work of Brown~\cite{zbMATH01731190}  we know how to  calculate~$\sigma|_{\partial\Omega}$ in terms of~$\Lambda_\sigma$, we can now also construct the  unique solution~$v$ to
%\begin{equation}\label{cond2}
%	\left\{
%		\begin{aligned}
%		\Delta v &=q v&  \text{in}&\ \Omega , \\
%		v|_{\partial\Omega} &= \sigma|_{\partial\Omega}^{1/2}.
%		\end{aligned}
%	\right.
%\end{equation}
%Then, recalling the definition \eqref{qog} of $q$, it is clear that $v=\sigma^{1/2}$ solves~\eqref{cond}, so that $$\sigma=v^2{\quad}  \text{in}\ \Omega.$$  This completes the reconstruction algorithm.}}

%%%%%%%%%%%%%%%%%%%% New section %%%%%%%%%%%%%%%%%%%%
\section{Sketch of the proof of the contraction for $\Delta^{\!-1}_\rho\circ {\rm M}_q$}\label{sketch}

 One of the main ideas of Haberman and Tataru \cite{zbMATH06145493} was to extend the domain of~$\Delta^{\!-1}_{\rho}$ using Bourgain-type spaces that are adapted to the problem, instead of the usual Sobolev spaces. With $s=1/2$ or $-1/2$, 
their norms 
 are defined by
\begin{equation*}
\label{map:normXb*}
\| \centerdot \|_{\dot{X}^s_\rho} : f \in \mathcal{S}(\R^n) \mapsto \big\| |m_\rho|^s \widehat{f}\, \big\|_{L^2(\R^n)},
\end{equation*}
where $m_\rho$ is the multiplier defined in \eqref{id:symbols}.  Then $\dot{X}^s_\rho$
is defined to be the Banach completion of $\mathcal{S}(\R^n)$
with respect to this norm. %Recalling that $m_\rho$ is quadratic, the first order derivatives are bounded when  $b=1/2$.
 It is immediate from the definitions that
\begin{equation}\label{id} \| \Delta^{\!-1}_{\rho} g \|_{\dot{X}^{1/2}_\rho} \le \| g \|_{\dot{X}^{-1/2}_\rho}
\end{equation}
whenever $g\in \mathcal{S}(\R^n)$, which can be used to continuously extend the operator. For ease of reference we will call  \eqref{id}   {\it the trivial inequality}.

On the other hand, Haberman and Tataru also proved  that ${\rm M}_q:f\mapsto qf$ satisfies
\begin{equation}\label{mult} \| {\rm M}_q f \|_{\dot{X}^{-1/2}_\rho} \le C\|\nabla\log \sigma\|_{\infty}(1+|\rho|^{-1}\|\nabla\log \sigma\|_{\infty})  \|f \|_{\dot{X}^{1/2}_\rho}
\end{equation}
whenever $f\in \dot{X}^{1/2}_\rho$; see \cite[Theorem 2.1]{zbMATH06145493}. Together these inequalities yield a contraction for $\Delta^{\!-1}_{\rho} \circ {\rm M}_q$ whenever $|\rho|>1$ and $\|\nabla\log \sigma\|_{\infty}$ is sufficiently small. In order to remove this smallness condition, we will alter the norms in such a way that the constant of \eqref{mult} can be taken small for any Lipschitz conductivity, while maintaining a version of~\eqref{id}.

 There is a natural  gain for the higher frequencies in \eqref{mult} whereas a gain for the lower frequencies can be engineered in \eqref{id} by introducing convex weights. This was the key observation of \cite{zbMATH06534426}. In order to have a gain for all frequencies, in at least one of the inequalities, we dampen  the higher frequencies relative to the lower frequencies in our main norm (with the lower frequencies dampened relative to the higher frequencies in the dual norm), so that the gain for the lower frequencies in our version of~\eqref{id} is passed through to our version of~\eqref{mult}. 

 We prove the Carleman estimate in Section~\ref{sec:Tzeta}, we define new Banach spaces in Section~\ref{spaces}, and then we extend the domain of $\Delta^{\!-1}_{\rho}$ via density in Section~\ref{trup}. We prove our version of~\eqref{mult} in Section~\ref{third} and then combine the estimates to obtain  the contraction in Section~\ref{contraction}.

%%%%%%%%%%%%%%%%%%%% New section %%%%%%%%%%%%%%%%%%%%
\section{Bounds for \texorpdfstring{$\Delta_{\rho}^{\!-1}$ with convex weights}{Tzeta}} \label{sec:Tzeta}

Let $B$ be an open ball centred at the origin, with radius $R:=2\sup_{x\in\Omega}|x|$, so that we comfortably have $\Omega\subset B$.
 The forthcoming constants will invariably depend on this~$R$, but never on the auxiliary parameters $\rho \in \C^n$ or $\lambda>1$. 
 
\subsection{The Carleman estimate:} Here we will deduce our estimate for $\Delta_{\rho}^{\!-1}$ from a Carleman estimate for $\Delta_{\rho}$, before defining  the main spaces and their duals in the following section. We improve upon the estimate
\begin{equation}\label{suboptimal}  |\rho|\| \Delta^{\!-1}_{\rho} f \|_{L^2(B)} \le C\| f \|_{L^2(\R^n)}   \end{equation}
whenever $f \in C^\infty_{\rm c}(B)$, which does not seem strong enough to 
construct CGO solutions for Lipschitz conductivities.
The inequality \eqref{suboptimal} follows by combining 
\begin{equation}\label{eleven}
|\rho|^{1/2} \|g\|_{L^2(B)} \leq C \|g\|_{\dot{X}^{1/2}_\rho}
\end{equation}
whenever $g\in \dot{X}^{1/2}_\rho$,  with the trivial inequality \eqref{id}, and then
\begin{equation}
\label{in:negativeXb}
|\rho|^{1/2}\| f \|_{\dot{X}^{-1/2}_\rho} \leq C\| f \|_{L^2(\R^n)}
\end{equation}
whenever $f \in C^\infty_{\rm c}(B)$. The constants $C>1$ depend only on $R$. Away from the zero set of the Fourier multiplier $m_\rho$ these inequalities are obvious, and the localisation serves to blur out the effect of the zero set; see Lemma~2.2 of~\cite{zbMATH06145493} for the proof. 

 In the following 
lemma we improve the constant in \eqref{suboptimal} by introducing exponential weights that depend on the auxiliary parameter $\lambda>1$. 
The extra gain in terms of~$\lambda$ will be key to construct our CGO solutions for Lipschitz 
conductivities.

\begin{lemma}\label{lem:L2B_boundedness} \sl Consider $\rho\in \mathbb{C}^n$ such that $\rho\cdot\rho= 0$  and write $\theta := \Re \rho / |\Re \rho|$. 
Then
%\begin{equation*}
%\| e^{\lambda (\theta \cdot {\rm x})^2/2} \Delta^{\!-1}_{\rho} f \|_{L^2 (B)} \leq \frac{2}{\lambda^{1/2} |\rho|} \| e^{\lambda (\theta \cdot {\rm x})^2/2} f \|_{L^2(\R^n)}
%\end{equation*}
\[ \int_{B}  | \Delta^{\!-1}_{\rho}f(x)|^2e^{\lambda (\theta \cdot x)^2}\dd x \leq \frac{2}{\lambda |\rho|^2} \int_{\R^n} |f(x)|^2e^{\lambda (\theta \cdot x)^2} \dd x \]
whenever  $f \in C^\infty_{\rm c}(\R^n)$ and $\lambda>1$ satisfies $|\rho| \geq 4 \lambda R$.
\end{lemma}

\begin{proof} If  $m_\rho$ had been defined slightly differently at the beginning, including a superfluous $\rho\cdot\rho$ term, we could have proved a version of this lemma without the hypothesis that $\rho\cdot\rho= 0$. In fact, we begin by reducing to a purely real vector case.  Indeed, letting $\Re \rho,\, \Im \rho\in \mathbb{R}^n$ denote the real and imaginary parts of~$\rho$, respectively, we define $\Delta^{\!-1}_{\Re \rho}$ as in Section~\ref{fad}, but with $m_\rho$ replaced by 
\begin{equation*}
m_{\Re \rho} (\xi): = -|\xi|^2 + 2i \Re \rho \cdot \xi+\Re \rho\cdot \Re \rho
\end{equation*}
for all $\xi\in \mathbb{R}^n$. Then, observing that 
$$
m_\rho(\xi)=-|\xi|^2 + 2i\rho \cdot \xi+\rho\cdot \rho=m_{\Re \rho}(\xi + \Im \rho), 
$$
and defining the modulation operator by ${\rm Mod}_{\Im \rho} f({\rm x}) := e^{i \Im \rho \cdot {\rm x}} f({\rm x})$, we find that
\[ {\rm Mod}_{\Im \rho} [\Delta^{\!-1}_{\rho} f] = \Delta^{\!-1}_{\Re \rho} [ {\rm Mod}_{\Im \rho} f] \]
whenever $f \in C^\infty_{\rm c}(\R^n)$. Recalling that $|\rho|^2=2|\Re \rho|^2$ if $\rho\cdot\rho=0$,  it will therefore suffice to prove 
%\[ \| e^{\lambda (\theta \cdot {\rm x})^2/2} \Delta^{\!-1}_{\Re \rho} f \|^2_{L^2 (B)} \leq \frac{1}{\lambda|\Re \rho|^2} \| e^{\lambda (\theta \cdot {\rm x})^2/2} f \|^2_{L^2(\R^n)} \]
\[ \int_{B}  | \Delta^{\!-1}_{\Re \rho}f(x)|^2e^{\lambda (\theta \cdot x)^2}\dd x \leq \frac{1}{\lambda |\Re \rho|^2} \int_{\R^n} |f(x)|^2e^{\lambda (\theta \cdot x)^2} \dd x \]
whenever $|\Re \rho|\ge 2 \lambda R$. Recalling that $\theta := \Re \rho / |\Re \rho|$, by  rotating to ${\rm e}_n$ this would follow from%if $Q_\theta $ belongs to the 
%orthogonal group $ \Ortho(n) $ and $\theta = Q_\theta e_n $, we have
%$ m_{\tau \theta} (Q_\theta \xi) = m_{\tau e_n} (\xi)$ for all $\xi \in \R^n$, so that
%\[ {\rm Q}_\theta [\Delta^{\!-1}_{\tau \theta} f] = \Delta^{\!-1}_{\tau e_n} [ {\rm Q}_\theta f] \]
%whenever  $f \in C^\infty_{\rm c}(\R^n)$,
%where ${\rm Q}_\theta f({\rm x}) = f(Q_\theta {\rm x})$. 
%Thus 
\begin{equation}
\label{in:boundednessTzeta}
\int_{B}  | \Delta^{\!-1}_{\tau {\rm e}_n}f(x)|^2e^{\lambda x_n^2}\,\dd x \leq \frac{1}{\lambda \tau^2} \int_{\R^n} |f(x)|^2e^{\lambda x_n^2}\, \dd x
\end{equation}
%\begin{equation}
%\label{in:boundednessTzeta}
%\| e^{\lambda {\rm x}_n^2/2} \Delta^{\!-1}_{\tau e_n} f \|^2_{L^2 (B)} \leq \frac{1}{\lambda \tau^2} \| e^{\lambda {\rm x}_n^2/2} f \|^2_{L^2(\R^n)}
%\end{equation}
whenever $f \in C^\infty_{\rm c}(\R^n)$ and $\tau \geq 2\lambda R$.

In order to prove \eqref{in:boundednessTzeta} we will first prove the closely related Carleman estimate 
\begin{equation}
\label{in:apriori_Delta_zeta}
 \| g \|^2_{L^2 (B)} \leq \frac{1}{\lambda\tau^2}\| e^{\lambda {\rm x}_n^2/2} (\Delta + 2 \tau {\rm e}_n \cdot \nabla + \tau^2) (e^{-\lambda {\rm x}_n^2/2} g) \|^2_{L^2(\R^n)}
\end{equation}
whenever $g \in \mathcal{S}(\R^n)$. Defining $\varphi (x) = \tau x_n + \lambda x_n^2/2$, the integrand of the right-hand side can be rewritten as
\begin{align*} %e^{\lambda {\rm x}_n^2/2} (\Delta + 2 \tau e_n \cdot \nabla + \tau^2) (e^{-\lambda {\rm x}_n^2/2} g)& 
e^{\varphi} \Delta (e^{-\varphi} g)= \Delta g - \nabla \varphi \cdot \nabla g - \nabla \cdot (\nabla \varphi g) + |\nabla \varphi|^2g. 
\end{align*}
Defining the formally self-adjoint ${\rm A}$ and skew-adjoint ${\rm B}$ by
\[{\rm A} g = \Delta g + |\nabla \varphi|^2 g\quad \text{and}\quad  {\rm B} g = - \nabla \varphi \cdot \nabla g - \nabla \cdot (\nabla \varphi g), \]
and integrating by parts, we have that
\begin{equation}
\label{id:sym-sk}
\| ({\rm A} + {\rm B}) g \|^2_{L^2(\R^n)} = \| {\rm A}g \|^2_{L^2(\R^n)} + \| {\rm B}g \|^2_{L^2(\R^n)} + \int_{\R^n} [{\rm A}, {\rm B}] g \overline{g} . 
\end{equation}
where $[{\rm A}, {\rm B}] = {\rm A}{\rm B} - {\rm B}{\rm A}$ denotes the commutator. By the definition of 
$\varphi$, we have
\begin{align*}
{\rm A}g(x) &= \Delta g(x) + (\tau + \lambda x_n)^2 g(x), \\
{\rm B}g(x) &= - 2(\tau + \lambda x_n) \partial_{x_n} g(x) - \lambda g(x),
\end{align*}
which yields
\[[{\rm A}, {\rm B}]g(x) = - 4 \lambda \partial^2_{x_n} g(x) + 4\lambda (\tau + \lambda x_n)^2 g(x).\]
After another integration by parts, we find
\[\int_{\R^n} [{\rm A}, {\rm B}] g \overline{g}  = 4 \lambda \int_{\R^n} |\partial_{x_n} g|^2  + 4 \lambda \int_{\R^n} |\nabla \varphi|^ 2 |g|^2 .\]
so that, substituting this into \eqref{id:sym-sk} and throwing three of the terms away, we find
\[\| e^{\varphi} \Delta (e^{-\varphi} g) \|^2_{L^2(\R^n)} \geq 4 \lambda \int_{\R^n} |\nabla \varphi|^ 2 |g|^2 . \]
As $| \nabla \varphi (x) | \geq \tau - \lambda R$
whenever $|x_n| \leq R$, this yields
\[\| e^{\varphi} \Delta (e^{-\varphi} g)\|^2_{L^2(\R^n)} \geq 4 \lambda (\tau - \lambda R)^2 \| g \|^2_{L^2 (B)}, \]
which implies \eqref{in:apriori_Delta_zeta} whenever
$\tau \geq 2\lambda R$  and 
$g \in \mathcal{S}(\R^n)$. 

Finally, by density, the inequality \eqref{in:apriori_Delta_zeta} also holds for 
every $g \in L^2_{\rm loc}(\R^n)$ such that $$e^{\lambda {\rm x}_n^2/2}  (\Delta + 2 \tau {\rm e}_n \cdot \nabla + \tau^2)(e^{-\lambda {\rm x}_n^2/2} g) \in L^2(\R^n).$$ Choosing
$g = e^{\lambda {\rm x}_n^2/2} \Delta^{\!-1}_{\tau {\rm e}_n} f$
with $f \in C^\infty_{\rm c} (\R^n)$ we find that 
\eqref{in:apriori_Delta_zeta} implies~\eqref{in:boundednessTzeta}.
\end{proof}

\begin{remark}  The proof of Lemma~\ref{lem:L2B_boundedness} yields the following strengthened estimate: If $\rho\in \mathbb{C}^n$ and $\theta := \Re \rho / |\Re \rho|$, 
then
\[ \int_{|\theta \cdot x| < \frac{|\Re \rho|}{2\lambda}}  | \Delta^{\!-1}_{\rho}f(x)|^2e^{\lambda (\theta \cdot x)^2}\dd x \leq \frac{1}{\lambda |\Re \rho|^2} \int_{\R^n} |f(x)|^2e^{\lambda (\theta \cdot x)^2} \dd x \]
whenever $f \in \mathcal{S}(\R^n)$ is such that the right-hand side is finite and $\lambda>1$. 
%is such that $\int_{\R^n} e^{\lambda (\theta \cdot {\rm x})^2} |f|^2 \, < \infty $. %It can also be strengthened in other ways; see~\cite[Section 4]{zbMATH06534426}.
\end{remark}

\subsection{Estimates for derivatives:} The inequality of Lemma~\ref{lem:L2B_boundedness} has a gain in the sense of~$L^2$, however this is not enough to construct CGO solutions for Lipschitz conductivities since we need to control the first-order partial 
derivatives present in the operator ${\rm M}_q$. For this we consider
\begin{equation}\label{defon}
\|\centerdot\|_{X_{\lambda,\rho}^{1/2}}:=\lambda^{1/4} |\rho|^{1/2} \| \centerdot \|_{L^2 (B,e^{\lambda (\theta \cdot {\rm x})^2})}+ \frac{1}{\lambda^{1/4}} \| \centerdot \|_{\dot{X}^{1/2}_\rho}
\end{equation}
and combine Lemma~\ref{lem:L2B_boundedness} with the trivial inequality \eqref{id}.

\begin{lemma}\label{lem:X1/2_L21}\sl
Consider $\rho\in \mathbb{C}^n$ such that $\rho\cdot\rho= 0$ and write $\theta := \Re \rho / |\Re \rho|$. Then there is a constant $C>1$, depending only on the radius $R$ of $B$, such that
\begin{equation*}
\| \Delta^{\!-1}_{\rho} f\|_{X_{\lambda,\rho}^{1/2}} \leq \frac{C}{\lambda^{1/4} |\rho|^{1/2}} \|f\|_{L^2 (\R^n\!,  e^{\lambda (\theta \cdot {\rm x})^2})}\end{equation*}
whenever $f \in C^\infty_{\rm c}(B)$ and $\lambda>1$ satisfies $|\rho| \geq 4 \lambda R$.
\end{lemma}

\begin{proof} The first term in the definition \eqref{defon} is bounded using Lemma~\ref{lem:L2B_boundedness}, so it remains to bound the second term. Combining the trivial inequality \eqref{id} with \eqref{in:negativeXb} we see that
\[ \| \Delta^{\!-1}_{\rho} f \|_{\dot{X}^{1/2}_\rho} \leq \| f \|_{\dot{X}^{-1/2}_\rho } \leq \frac{C}{|\rho|^{1/2}} \| f \|_{L^2(\R^n)}\le \frac{C}{|\rho|^{1/2}} \| f \|_{L^2(\R^n\!, e^{\lambda (\theta \cdot {\rm x})^2})} \]
whenever $f\in C^\infty_{\rm c}(B)$, where the constant $C>1$ depends only on $R$.
Dividing by~$\lambda^{1/4}$ yields the desired estimate for the second term.\end{proof}

\begin{lemma}\label{lem:L2B_X-1/2} \sl
Consider $\rho\in \mathbb{C}^n$ such that $\rho\cdot\rho= 0$ and  write $\theta := \Re \rho / |\Re \rho|$. Then there is a constant $C>1$, depending only on the radius $R$ of $B$, such that
\begin{equation*}
\| \Delta^{\!-1}_{\rho} f \|_{X_{\lambda,\rho}^{1/2}} \leq C\lambda^{1/4}e^{\lambda R^2/2} \| f \|_{\dot{X}^{-1/2}_\rho}
\end{equation*}
whenever $f \in C^\infty_{\rm c}(B)$ and 
$\lambda>1$.
\end{lemma}

\begin{proof} The second term in the definition \eqref{defon} can be bounded easily using the trivial inequality~\eqref{id}, so it remains to bound the first term.
By  \eqref{eleven},  we have
\begin{equation*}
|\rho|^{1/2}\|g \|_{L^2(B, e^{\lambda (\theta \cdot {\rm x})^2})} \leq  e^{\lambda R^2/2}|\rho|^{1/2}\| g \|_{L^2(B)}\leq Ce^{\lambda R^2/2}\| g \|_{\dot{X}^{1/2}_\rho}\end{equation*}
whenever $g\in \dot{X}^{1/2}_\rho$, where the constant $C>1$ depends only on $R$.  Taking $g=\Delta^{\!-1}_{\rho} f$ and multiplying the inequality by $\lambda^{1/4}$  yields
$$
\lambda^{1/4}|\rho|^{1/2}\|\Delta^{\!-1}_{\rho} f \|_{L^2(B, e^{\lambda (\theta \cdot {\rm x})^2})}  \le  C\lambda^{1/4}e^{\lambda R^2/2}\| \Delta^{\!-1}_{\rho} f\|_{\dot{X}^{1/2}_\rho},
$$
 A final  application of the trivial inequality \eqref{id} yields the desired estimate.
\end{proof}

%%%%%%%%%%%%%%%%%%%% New section %%%%%%%%%%%%%%%%%%%%
\section{The new spaces} \label{spaces}

We must extend the domain of $\Delta^{\!-1}_{\rho}$ by taking limits and so we carefully define Banach spaces using equivalence classes. 
 We define 
$$\dot{X}^{1/2}_\rho (B) := \big\{ [f]_B : f \in \dot{X}^{1/2}_\rho \big\},$$ where the equivalence class $[f]_B$ is given by
$$[f]_B := \big\{ g \in \dot{X}^{1/2}_\rho\, :\, {\rm ess}\supp (f  - g) \subset \R^n \setminus B \big\}.$$ 
The space can be
endowed with the norm
\[ \big\| [f]_B \big\|_{\dot{X}^{1/2}_\rho (B)}: = \inf \big\{ \| g\|_{\dot{X}^{1/2}_\rho} : g \in [f]_B \big\},\]
so that 
$$\Big(\dot{X}^{1/2}_\rho (B), \| \centerdot \|_{\dot{X}^{1/2}_\rho (B)} \Big)\ \
\text{is a Banach space.}$$ 
We can rephrase the inequality \eqref{eleven} in terms of this norm. Indeed as
\begin{equation*}
|\rho|^{1/2}\|g \|_{L^2(B)} \le C\| g \|_{\dot{X}^{1/2}_\rho}
\end{equation*}
whenever $g\in [f]_B$, where $C > 1$ is a constant depending only on $R$, we can take the infimum to find \begin{equation}\label{23}
|\rho|^{1/2} \| f \|_{L^2(B)} \le  C\| f \|_{\dot{X}^{1/2}_\rho(B)}.
\end{equation}
Identifying the elements $[f]_{B}$ of $\dot{X}^{1/2}_{\rho} ({B})$ with $f|_{B}$, the restriction of $f$ to~${B}$,
 this yields the embedding
\begin{equation}\label{embed}\dot{X}^{1/2}_\rho (B) \hookrightarrow L^2(B).\end{equation}
Moreover, we have the following equivalence of norms. 

\subsection{Equivalence with the Sobolev norm:} There are constants $c,C>0$, depending only on $R$, such that \begin{equation}\label{equivall}
c|\rho|^{-1/2}\|f\|_{H^1(B)}\le \|f\|_{\dot{X}^{1/2}_\rho(B)}\le C|\rho|^{1/2}\|f\|_{H^1(B)}
\end{equation}
whenever $f\in H^1(B)$ and $|\rho|>1$. To see this, note that $|m_\rho(\xi)|\le 2(1+|\rho|)(1+|\xi|^2)$  for all $\xi\in \mathbb{R}^n$, so that
\begin{equation*}
\|g\|_{\dot{X}^{1/2}_\rho}\le 2^{1/2}(1+|\rho|)^{1/2}\|g\|_{H^1(\mathbb{R}^n)}
\end{equation*}
whenever   $g\in H^1(\mathbb{R}^n)$. Thus $H^1$-extensions are also $\dot{X}^{1/2}_\rho$-extensions, so the right-hand inequality of~\eqref{equivall} follows by taking the infimum over $H^1$-extensions $g$ of $f\in H^1(B)$.

For the left-hand inequality, consider $g_{B}:=\chi_Bg$, where $\chi_B$ is a smooth function equal to one on $B$ and supported on $2B$. Then, separating the low and high frequencies, 
\begin{align*}
\|g_{B}\|^2_{H^1(\mathbb{R}^n)} 
%&= \int_{|\xi|\le 4|\rho|} (1+|\xi|^2)|\widehat{g_{B}}(\xi)|^2\, \dd \xi+ \int_{|\xi|> 4|\rho|} (1+|\xi|^2)|\widehat{g_{B}}(\xi)|^2\, \dd \xi\\
&\le \|g_{B}\|^2_{L^2(\mathbb{R}^n)}+16|\rho|^2\!\!\int_{|\xi|\le 4|\rho|} |\widehat{g_{B}}(\xi)|^2\, \dd \xi+ 2\int_{|\xi|> 4|\rho|} |m_\rho(\xi)||\widehat{g_{B}}(\xi)|^2\, \dd \xi\\
&\le C|\rho|\|g\|^2_{\dot{X}_{\rho}^{1/2}}
\end{align*}
whenever $|\rho|>1$, where the second inequality follows from Lemma~2.2 of~\cite{zbMATH06145493}. Restricting the left-hand side to $B$, we find that
\begin{align*}\nonumber
\|g\|_{H^1(B)}\le C|\rho|^{1/2}\|g\|_{\dot{X}_{\rho}^{1/2}}.
\end{align*}
Now if $g$ is an $\dot{X}^{1/2}_\rho$-extension of $f\in \dot{X}^{1/2}_\rho(B)$, then $f=g$ almost everywhere in $B$ and so we can replace $g$ on the left-hand side by~$f$ and take the infimum over $g$ to obtain the left-hand inequality of \eqref{equivall}.

\subsection{The main space:} %Similarly we have that
 %\begin{equation*}\label{this2}
%|\rho|^{1/2} \| f \|_{L^2(B,e^{\lambda (\theta \cdot {\rm x})^2})} \le  Ce^{\lambda R^2/2}\| f \|_{\dot{X}^{1/2}_\rho(B)}.
%\end{equation*}
We define our main norm by
\vspace{6pt}
\begin{mdframed}[style=MyFrame]
\[ \| \centerdot \|_{X^{1/2}_{\lambda, \rho} ({B})} : f \in \dot{X}^{1/2}_\rho ({B}) \mapsto \lambda^{1/4} |\rho|^{1/2} \| f \|_{L^2 ({B},e^{\lambda (\theta \cdot {\rm x})^2})} + \frac{1}{\lambda^{1/4}} \| f \|_{\dot{X}^{1/2}_\rho ({B})}, \]
\end{mdframed}
\vspace{2pt}
and note that by \eqref{23} it is equivalent to the inhomogeneous norm;
\begin{align}\label{equiv}
\lambda^{-1/4}\| f\|_{\dot{X}^{1/2}_{\rho}({B})}\le \| f\|_{X^{1/2}_{\lambda,\rho}({B})}\le C\lambda^{1/4}e^{\lambda R^2/2}\| f\|_{\dot{X}^{1/2}_{\rho}({B})},
\end{align}
where $C>1$ depends only on $R$.
Thus we can conclude that \begin{equation}\label{X}\Big(\dot{X}^{1/2}_\rho ({B}), \| \centerdot \|_{X^{1/2}_{\lambda, \rho} ({B})} \Big)\ 
\text{is a Banach space.}\end{equation}
Later we will use that the constants in this norm equivalence are independent of~$|\rho|$. 

\subsection{A minor variant of the main space:} We also consider the  norm $\| \centerdot\|_{Y^{1/2}_{\lambda, -\rho}({B})}$ defined by
\[  f \in \dot{X}^{1/2}_{-\rho} ({B}) \mapsto \max\Big\{\lambda^{1/4} |\rho|^{1/2} \|f\|_{L^2 ({B}, e^{- \lambda (\theta \cdot {\rm x})^2} )},\frac{1}{\lambda^{1/4}e^{\lambda R^2/2}} \|f\|_{\dot{X}^{1/2}_{-\rho}(B)}\Big\}. \]
Notice that little more than some signs have changed. As before, this norm is equivalent to the inhomogeneous norm;
\begin{align}\label{equivY}
\frac{1}{\lambda^{1/4}e^{\lambda R^2/2}} \| f\|_{\dot{X}^{1/2}_{-\rho}({B})}\le \| f\|_{Y^{1/2}_{\lambda,-\rho}({B})}\le C\lambda^{1/4}\| f\|_{\dot{X}^{1/2}_{-\rho}({B})},
\end{align}
where $C>1$ depends only on $R$,  and so
\begin{equation}
\label{term:Y1/2}
\Big(\dot{X}^{1/2}_{-\rho} ({B}), \| \centerdot \|_{Y^{1/2}_{\lambda, -\rho} ({B})}\Big)\ \ \text{is a Banach space.}
\end{equation}
Recalling the embedding \eqref{embed}, this can be identified with the intersection of the spaces
\[\Big(L^2 ({B}), \lambda^{1/4} |\rho|^{1/2} \| \centerdot \|_{L^2 ({B},e^{-\lambda (\theta \cdot {\rm x})^2})}\Big)\quad 
\text{and}
\quad\Big(\dot{X}^{1/2}_{-\rho} ({B}), \frac{1}{\lambda^{1/4}e^{\lambda R^2/2}}\| \centerdot \|_{\dot{X}^{1/2}_{-\rho} ({B})}\Big), \]
As \eqref{term:Y1/2} is dense in both of these spaces, we can identify the dual of their intersection with the sum of their duals; see for example~\cite[Theorem 3.1]{zbMATH03457577}. This provides an alternative identification of the dual of \eqref{term:Y1/2} which we describe now.

 \subsection{The dual space:} Let $\dot{X}^{-1/2}_{\rho, {\rm c}}({B})$ denote the Banach completion of $C^\infty_{\rm c}({B})$ with respect to the norm
\[ \| \centerdot \|_{\dot{X}^{-1/2}_\rho} : f \in C^\infty_{\rm c}({B}) \mapsto \big\| |m_\rho|^{-1/2} \widehat{f}\, \big\|_{L^2(\R^n)}. \] 
We endow $L^2({B}) + \dot{X}^{-1/2}_{\rho, {\rm c}}({B})$ with the norm
 \vspace{6pt}
\begin{mdframed}[style=MyFrame]
\[\| f \|_{Y^{-1/2}_{\lambda, \rho, {\rm c}} ({B})} := \inf_{f = f^\flat + f^\sharp} \Big( \frac{1}{\lambda^{1/4} |\rho|^{1/2}} \|  f^\flat \|_{L^2(B, e^{\lambda (\theta \cdot {\rm x})^2})} + \lambda^{1/4} e^{\lambda R^2/2} \| f^\sharp \|_{\dot{X}^{-1/2}_\rho}  \Big) \]
\end{mdframed}
\vspace{2pt}
with the infimum taken over all $f^\flat \in L^2({B})$ and $f^\sharp \in \dot{X}^{-1/2}_{\rho, {\rm c}}({B})$.
Then
   \begin{equation}
\label{term:Y-1/2}
\Big( L^2({B}) + \dot{X}^{-1/2}_{\rho, {\rm c}}({B}), \| \centerdot \|_{Y^{-1/2}_{\lambda, \rho, {\rm c}} ({B})}\Big)\ \
\text{is a Banach space.}
\end{equation}
With real-bracket pairings, Plancherel's identity takes the form
\begin{equation}\label{planchar}
\langle f, g \rangle=\int_{\R^n} \widehat{f}(\xi)g^\vee(\xi)\,\dd\xi =\int_{\R^n} \widehat{f}(\xi)\widehat{g}(-\xi)\,\dd\xi, 
\end{equation}
so that, by similar arguments to those used for Sobolev spaces,  we find that
$$
\Big(\dot{X}^{1/2}_{-\rho} ({B}), \frac{1}{\lambda^{1/4}e^{\lambda R^2/2}}\| \centerdot \|_{\dot{X}^{1/2}_{-\rho} ({B})}\Big)^\ast\cong\Big(\dot{X}^{-1/2}_{\rho,{\rm c}} ({B}), \lambda^{1/4}e^{\lambda R^2/2}\| \centerdot \|_{\dot{X}^{-1/2}_{\rho}}\Big);
$$
see for example \cite[Proposition 2.9]{zbMATH00764042}. On the other hand, it is easy to see that
$$
\Big(L^2 ({B}), \lambda^{1/4} |\rho|^{1/2} \| \centerdot \|_{L^2 ({B},e^{-\lambda (\theta \cdot {\rm x})^2})}\Big)^\ast\cong\Big(L^2 ({B}), \frac{1}{\lambda^{1/4} |\rho|^{1/2}} \| \centerdot \|_{L^2 ({B},e^{\lambda (\theta \cdot {\rm x})^2})}\Big).
$$
Thus the dual of \eqref{term:Y1/2} can be identified with  the sum of the two dual spaces as described in \eqref{term:Y-1/2}; see for example \cite[Theorem 3.1]{zbMATH03457577}.

%%%%%%%%%%%%%%%%%%%%%%%%%%%%%%%%%%%%%%%%%%%%%%%%%%%%%
\section{The locally defined extension of $\Delta^{\!-1}_{\rho}$}\label{trup} 
We are now ready to extend the domain of  $\Delta^{\!-1}_{\rho}$ by combining Lemmas~\ref{lem:X1/2_L21} and~\ref{lem:L2B_X-1/2}. This extension will make no sense outside of $B$ in contrast with the globally defined extension of $f\in C^\infty_{\rm c}(B)\mapsto \Delta^{\!-1}_{\rho} f$ given by the trivial inequality~\eqref{id}. We denote the globally defined extension by $\Delta^{\!-1}_{\rho}$ and the locally defined extension by~${\rm T}^B_{\!\rho}$.

\begin{corollary}\label{cor:boundedness} \sl
Consider $\rho\in \mathbb{C}^n$ such that $\rho\cdot\rho =0$ and $ \lambda>1$. 
Then
 there is a continuous linear extension ${\rm T}^B_{\!\rho}$ of $$f\in C^\infty_{\rm c}(B)\mapsto \Delta^{\!-1}_{\rho} f|_{B}$$ and a constant $C>1$, depending only on the radius $R$ of $B$, such that
\[ \| {\rm T}^B_{\!\rho} f \|_{X^{1/2}_{\lambda, \rho}(B)} \leq C \| f \|_{Y^{-1/2}_{\lambda, \rho, {\rm c}}(B)}
\]
whenever $f \in L^2(B) + \dot{X}^{-1/2}_{\rho, {\rm c}}(B)$ and
$|\rho| \geq 4 \lambda R$.
\end{corollary}

\begin{proof} By Lemma~\ref{lem:X1/2_L21} and the 
density of $C^\infty_{\rm c}({B})$ in
\begin{equation*}
\label{sp:L2cB}
\Big(L^2({B}), \frac{1}{\lambda^{1/4} |\rho|^{1/2}} \|  \centerdot \|_{L^2 (B, e^{\lambda (\theta \cdot {\rm x})^2})}\Big),
\end{equation*}
we can extend $f\in C^\infty_{\rm c}(B)\mapsto \Delta^{\!-1}_{\rho} f|_{B}$ to a bounded linear operator ${\rm T}^B_{\!\rho}$ that satisfies 
\[\|{\rm T}^B_{\!\rho}  f\|_{X^{1/2}_{\lambda, \rho}({B})}  \le \frac{C}{\lambda^{1/4} |\rho|^{1/2}} \|  f \|_{L^2(B, e^{\lambda (\theta \cdot {\rm x})^2})} \]
whenever $f \in L^2({B})$. The constant $C>1$ depends only on $R$. On the other hand, by Lemma~\ref{lem:L2B_X-1/2} and the density of $C^\infty_{\rm c}({B})$ in
\begin{equation*}
\label{sp:X-1/2c}
\Big(\dot{X}^{-1/2}_{\rho, {\rm c}}({B}),  \lambda^{1/4} e^{\lambda R^2/2} \| \centerdot \|_{\dot{X}^{-1/2}_\rho}\Big),
\end{equation*}
we can extend $f\in C^\infty_{\rm c}(B)\mapsto \Delta^{\!-1}_{\rho} f|_{B}$ to a bounded linear operator ${\rm T}^B_{\!\rho}$ that satisfies
\[ \|{\rm T}^B_{\!\rho}  f\|_{X^{1/2}_{\lambda, \rho}({B})} \le  C\lambda^{1/4} e^{\lambda R^2/2} \| f \|_{\dot{X}^{-1/2}_\rho} \]
whenever $f \in \dot{X}^{-1/2}_{\rho, {\rm c}}({B})$.
Again, the constant $C>1$ depends only on $R$.

Considering now $f = f^\flat + f^\sharp$
with $f^\flat \in L^2({B})$ and $f^\sharp \in \dot{X}^{-1/2}_{\rho, {\rm c}}({B})$, we define
$${\rm T}^B_{\!\rho} f := {\rm T}^B_{\!\rho}  f^\flat + {\rm T}^B_{\!\rho}  f^\sharp.$$ One can show that this is well-defined using the linearity of the previous extensions and the  density of $C^\infty_{\rm c}(B)$. Then, by the triangle inequality and the previous bounds,
\[ \| {\rm T}^B_{\!\rho} f \|_{X^{1/2}_{\lambda, \rho}(B)} \leq C\Big( \frac{1}{\lambda^{1/4} |\rho|^{1/2}} \| f^\flat \|_{L^2(e^{\lambda (\theta \cdot {\rm x})^2})} + \lambda^{1/4} e^{\lambda R^2/2} \| f^\sharp \|_{\dot{X}^{-1/2}_\rho}\Big), \]
where the constant $C$ depends only on $R$. Since the left-hand side is independent of the representation $f= f^\flat + f^\sharp$, we can 
take the infimum over such representations, and the desired inequality follows.
\end{proof}

%%%%%%%%%%%%%%%%%%%% New section %%%%%%%%%%%%%%%%%%%%
\section{The bound for ${\rm M}_q$}\label{third}

With a view to further applications, we write part of this section in greater generality. Consider bounded functions $ a_0, a_1, \dots , a_n \subset L^\infty (\mathbb{R}^n)$ 
with compact support;
\begin{equation*}
\label{term:support_cond}
 \supp a_j  \subset \Omega \subset B=\{ x \in \R^n : |x| < R \},
\end{equation*}
where $R:=2\sup_{x\in\Omega}|x|$. Define the bilinear form $\mathcal{B}:H^1(B)\times H^1(B)\to \mathbb{C}$ by
\[ \mathcal{B}(f , g):=\int_{\Omega} a_0 f g \, +  \int_{\Omega} A \cdot \nabla (f  g) , \]
where $A$ is the vector field with components $(a_1, \dots , a_n)$. This is well-defined by an application of the product rule, followed by the Cauchy--Schwarz inequality.

\begin{proposition}\label{prop:boundednessBILINEAR}\sl
Consider   $\rho \in \C^n$ such 
that $\rho\cdot\rho=0$ and $ \lambda>1$. Then there is a constant $C>1$, depending only on the radius $R$ of $B$, such that
\[|\mathcal{B}(f , g)| \le  C\Big( \frac{1}{\lambda^{1/2}|\rho|}+\frac{1}{\lambda^{1/2}} + \frac{e^{\lambda R^2/2}}{|\rho|^{1/2}}  \Big) \sum_{j=0}^n \| a_j \|_{L^\infty(\Omega)} \| f \|_{X^{1/2}_{\lambda, \rho} ({B})} \| g \|_{Y_{\lambda, -\rho}^{1/2} ({B})} \]
whenever $(f ,g) \in \dot{X}_{\rho}^{1/2}(B) \times \dot{X}_{-\rho}^{1/2}(B)$. 
\end{proposition}

\begin{proof} %Once we prove the estimate for $(f ,g) \in C^\infty(B) \times C^\infty(B)$, the domain of the bilinear form can be extended to $\dot{X}^{1/2}_{\rho} ({B}) \times \dot{X}^{1/2}_{\rho} ({B})$ by density. 
For the first term, we note that by the Cauchy--Schwarz inequality,
\begin{equation*}
\label{in:a_0}
\begin{aligned}
\Big| \int_{\Omega} a_0 fg \, \Big| & \leq \| a_0 \|_{\infty} \| e^{\lambda (\theta \cdot {\rm x})^2/2} f\|_{L^2({B})} \| e^{- \lambda (\theta \cdot {\rm x})^2/2} g \|_{L^2({B})}\\
& \leq \frac{1}{\lambda^{1/2}} \frac{1}{|\rho|} \| a_0 \|_{\infty} \|f \|_{X^{1/2}_{\lambda, \rho}({B})} \|g \|_{Y_{\lambda, -\rho}^{1/2}({B})}
\end{aligned}
\end{equation*}
whenever $(f,g)\in \dot{X}_{\rho}^{1/2}(B) \times \dot{X}_{-\rho}^{1/2}(B)$.
The second inequality follows directly from the weightings in the definition of the norms. 

For the more difficult first order term, we consider a positive and smooth function~$\chi$, equal to one on the  ball of radius $1/2$, supported in the unit ball, and bounded above by one. Then we work with $f_{\!B}:=\chi_Bf$ and $g_{B}:=\chi_Bg$, where $\chi_B:=\chi(\,\centerdot\,/R)$ is  equal to one on $\Omega$ and  supported on $B$. Letting $A^\flat $ denote the vector field with components
\[ a^\flat_j (x) := \frac{1}{(2 \pi)^{n}} \int_{\R^n} e^{i x \cdot \xi} \chi \Big( \frac{\xi}{16|\rho|} \Big) \widehat{a}_j (\xi) \, \dd \xi \]
 for all $x \in \R^n$ and $j =1, \dots , n$, and letting $A^\sharp := A - A^\flat$,  by integration by parts,
\[ \int_{\Omega} A \cdot \nabla (f g) \, =  - \int_{\R^n} \nabla \cdot A^\flat \, f_{\!B}g_{B} \, + \int_{\R^n} A^\sharp \cdot \nabla (f_{\!B}g_{B}). \]
Noting that $\|\nabla \cdot A^\flat\|_\infty\leq C |\rho| \|A\|_\infty$, the first term can be bounded as before;
\begin{align*}\nonumber\Big| \int_{\R^n} \nabla \cdot A^\flat \, f_{\!B}g_{B} \, \Big| &\leq C \| \nabla \cdot A^\flat \|_{\infty} \| e^{\lambda (\theta \cdot {\rm x})^2/2} f \|_{L^2(B)} \| e^{-\lambda (\theta \cdot {\rm x})^2/2} g \|_{L^2(B)}\\
&\leq C \| A\|_{\infty}  |\rho|^{1/2}\|f\|_{L^2(B, e^{\lambda (\theta \cdot {\rm x})^2})}|\rho|^{1/2}\|g \|_{L^2(B, e^{-\lambda (\theta \cdot {\rm x})^2})}.
\end{align*}
Again by the weightings in the definitions of the norms, this implies that
\begin{align}\nonumber\Big| \int_{\R^n} \nabla \cdot A^\flat \, f_{\!B}g_{B} \, \Big| &\leq C  \frac{1}{\lambda^{1/2}} \| A\|_{\infty} \|f \|_{X^{1/2}_{\lambda, \rho}(B)} \|g\|_{Y_{\lambda,- \rho}^{1/2}(B)}\label{in:Asharp}
\end{align}
whenever $(f,g)\in \dot{X}_{\rho}^{1/2}(B) \times \dot{X}_{-\rho}^{1/2}(B)$.

It remains to show that
\begin{equation}\label{end}
\Big|\int_{\R^n} A^\sharp \cdot \nabla (f_{\!B}g_{B})\Big|\leq C \frac{e^{\lambda R^2/2}}{|\rho|^{1/2}}  \| A \|_{\infty} \| f \|_{X^{1/2}_{\lambda, \rho} ({B})} \| g \|_{Y_{\lambda,-\rho}^{1/2} ({B})}.
\end{equation}
Using the product rule, we can separate into two similar terms;
\begin{equation}\label{anoth}
\int_{\R^n} A^\sharp \cdot \nabla (f_{\!B}g_{B}) = \int_{\R^n} A^\sharp \cdot \nabla f_{\!B}g_{B} \, + \int_{\R^n} A^\sharp \cdot \nabla g_{B} f_{\!B},
\end{equation}
and initially treat the first term on the right-hand side (the second term will eventually be dealt with by symmetry). 
We decompose the integral as
\[ \int_{\R^n} A^\sharp \cdot \nabla f_{\!B}g_{B} \, = \int_{\R^n} A^\sharp \cdot \nabla L f_{\!B} L g_{B} \, + \int_{\R^n} A^\sharp \cdot \nabla L f_{\!B} H g_{B} \, + \int_{\R^n} A^\sharp \cdot \nabla H\! f_{\!B}g_{B}, \]
 where $L$ denotes the low-frequency filter defined by
\[ L f({\rm x}) := \frac{1}{(2 \pi)^{n}} \int_{\R^n} e^{i {\rm x} \cdot \xi} \chi \Big( \frac{\xi}{4 |\rho|} \Big) \widehat{f} (\xi) \, \dd \xi
\]
and $H:=I-L$.
By the properties of $\chi$, the frequency supports of $\nabla L f_{\!B} L g_{B}$ and $A^\sharp$ are disjoint, so that by Plancherel's identity the first term is in fact zero, yielding
\[ \int_{\R^n} A^\sharp \cdot \nabla f_{\!B}g_{B} \, = \int_{\R^n} A^\sharp \cdot \nabla L f_{\!B} H g_{B} \, + \int_{\R^n} A^\sharp \cdot \nabla H\! f_{\!B}g_{B}.\]
Then, by the Cauchy--Schwarz inequality (writing $\|\centerdot\|_2:=\|\centerdot\|_{L^2(\R^n)}$),
\begin{align*}
\Big| \int_{\R^n} A^\sharp \cdot \nabla f_{\!B}g_{B} \, \Big| \leq \| A^\sharp \|_{\infty} \Big(  \| \nabla L f_{\!B}& \|_{2} \| H g_{B} \|_{2}  + \| \nabla H\! f_{\!B} \|_{2} \| g_{B} \|_{2} \Big).
\end{align*}
Now as $\| A^\sharp \|_{\infty}\leq C\| A\|_{\infty}$ and
\begin{align*}
\| \nabla L f_{\!B} \|_{2} \| H g_{B} \|_{2} & \leq C |\rho| \| L f_{\!B} \|_{2} \| H g_{B} \|_{2} \leq C \| f_{\!B} \|_{2} \| \nabla H g_{B} \|_{2},
\end{align*}
 we find that
\begin{align*}
\Big| \int_{\R^n} A^\sharp \cdot \nabla f_{\!B}g_{B} \, \Big| \leq C \| A \|_{\infty} \Big(\| f_{\!B} \|_{2} &\| \nabla H g_{B} \|_{2}+ \| \nabla H\! f_{\!B} \|_{2} \| g_{B} \|_{2} \Big).
\end{align*}
Since the right-hand side is symmetric in the roles of $f_{\!B}$ and $g_{B}$, we can conclude the same bound for the second term on the right-hand side of~\eqref{anoth}, yielding
\begin{align}\label{another}
\Big| \int_{\R^n} A^\sharp \cdot \nabla (f_{\!B}g_{B}) \, \Big| \leq C \| A \|_{\infty} \Big( \| f_{\!B} \|_{2} &\| \nabla H g_{B} \|_{2}+ \| \nabla H\! f_{\!B} \|_{2} \| g_{B} \|_{2} \Big).
\end{align}
Now clearly we have that
$$\| f_{\!B} \|_{2}\le \|f\|_{L^2(B,e^{\lambda (\theta \cdot {\rm x})^2})}\quad \text{and}\quad \| g_{B} \|_{2}\le e^{\lambda R^2/2}\|g\|_{L^2(B,e^{-\lambda (\theta \cdot {\rm x})^2})}.$$
 On the other hand, by Lemma~2.2 of \cite{zbMATH06145493}, we have
\[ \| \nabla H\! f_{\!B} \|_{2} \leq C \| \widetilde{f} \|_{\dot{X}_{\rho}^{1/2}}\quad \text{and}\quad \| \nabla H g_{B} \|_{2} \leq C \| \widetilde{g} \|_{\dot{X}_{-\rho}^{1/2}}\]
where $(\widetilde{f},\widetilde{g}) \in \dot{X}_{\rho}^{1/2} \times \dot{X}_{-\rho}^{1/2}$ denotes any pair of extensions of $(f,g)$.
Substituting these inequalities into \eqref{another}, and taking the infimum over extensions, yields 
\begin{align*}
\Big| \int_{\R^n} A^\sharp \cdot \nabla (f_{\!B}g_{B}) \, \Big| \leq C  \| A \|_{\infty} \Big(\|f&\|_{L^2(B,e^{\lambda (\theta \cdot {\rm x})^2})}\|g \|_{\dot{X}^{1/2}_{-\rho}(B)}\\
&+ e^{\lambda R^2/2}\|f \|_{\dot{X}^{1/2}_{\rho}(B)}\|g\|_{L^2(B,e^{-\lambda (\theta \cdot {\rm x})^2})} \Big).
\end{align*}
Recalling the weightings in the the norms, this completes the proof of~\eqref{end}.
\end{proof}

From this we can deduce our estimate for ${\rm M}_q:f\mapsto qf$, where $q$ is defined in~\eqref{qog}.

\begin{corollary}\label{cor:M_V}\sl
Consider $\rho \in \C^n$ such that  $\rho\cdot\rho=0$ and $\lambda>1$.  Then there is a $C>1$, depending on $\|\nabla\log \sigma\|_{\infty}$ and the radius $R$ of $B$, such that
\[ \| {\rm M}_qf \|_{Y_{\lambda, \rho, {\rm c}}^{-1/2}({B})} \le C\Big( \frac{1}{\lambda^{1/2}|\rho|}+\frac{1}{\lambda^{1/2}} + \frac{e^{\lambda R^2/2}}{|\rho|^{1/2}}  \Big)\|f \|_{X^{1/2}_{\lambda, \rho} ({B})} \]
whenever $f \in \dot{X}_{\rho}^{1/2}({B})$. 
\end{corollary}

\begin{proof} By an application of the product rule, the definition \eqref{qog} can be rewritten as
\begin{align*}
\big\langle q, \psi \big\rangle &= \frac{1}{4} \int_{\Omega} |\nabla \log\sigma|^2  \psi \, - \frac{1}{2} \int_{\Omega} \nabla\log \sigma \cdot \nabla \psi.
\end{align*}
Using our {\it a priori} assumptions that $\sigma$ is bounded below and $\nabla \sigma$ is bounded above almost everywhere (which follows from Lipschitz continuity), $\nabla\log\sigma=\sigma^{-1}\nabla \sigma$ is a vector of bounded functions. Thus, by taking $$a_0=\frac{1}{4}|\nabla\log\sigma|^2\quad  \text{and}\quad  A=-\frac{1}{2}\nabla\log \sigma,$$ we can write $ \langle {\rm M}_q f, g \rangle :=\langle q f, g \rangle:=\langle q , fg \rangle = \mathcal{B} (f, g) $ for all 
$(f,g)\in \dot{X}_{\rho}^{1/2}(B) \times \dot{X}_{-\rho}^{1/2}(B)$. Then, by Proposition~\ref{prop:boundednessBILINEAR} we find that
\[|\langle {\rm M}_q f, g \rangle| \leq C \Big(\frac{1}{\lambda^{1/2}|\rho|}+ \frac{1}{\lambda^{1/2}} + \frac{e^{\lambda R^2/2}}{|\rho|^{1/2}}  \Big) \sum_{j=0}^n \| a_j \|_{\infty} \| f \|_{X^{1/2}_{\lambda, \rho} ({B})} \| g\|_{Y_{\lambda, -\rho}^{1/2} ({B})} \]
for all $(f,g)\in \dot{X}_{\rho}^{1/2}(B) \times \dot{X}_{-\rho}^{1/2}(B)$. Finally, using the identification
$$\Big(\dot{X}^{1/2}_{-\rho} ({B}), \| \centerdot \|_{Y^{1/2}_{\lambda, -\rho} ({B})}\Big)^\ast\cong \Big( L^2({B}) + \dot{X}^{-1/2}_{\rho, {\rm c}}({B}), \| \centerdot \|_{Y^{-1/2}_{\lambda, \rho, {\rm c}} ({B})}\Big),$$
we obtain the desired inequality. 
\end{proof}

%%%%%%%%%%%%%%%%%%%% New section %%%%%%%%%%%%%%%%%%%%
\section{Locally defined CGO solutions via Neumann series}\label{contraction}

Let $X^{1/2}_{\lambda, \rho} ({B})$ and $Y^{-1/2}_{\lambda, \rho, {\rm c}} ({B})$ denote the Banach spaces defined in \eqref{X}  and \eqref{term:Y-1/2}, respectively.  Recall that $f\in C^\infty_{\rm c}(B)\mapsto \Delta^{\!-1}_{\rho} f|_{B}$ can be extended as a bounded linear operator
$${\rm T}^B_{\!\rho}: Y^{-1/2}_{\lambda, \rho, {\rm c}} ({B}) \to  X^{1/2}_{\lambda, \rho} ({B}),$$
using Corollary~\ref{cor:boundedness}, and that  ${\rm M}_q: f\mapsto qf$, with $q$ defined in \eqref{qog},  is bounded as
$${\rm M}_q: X^{1/2}_{\lambda, \rho} ({B}) \to  Y^{-1/2}_{\lambda, \rho, {\rm c}} ({B}),$$
by Corollary~\ref{cor:M_V}. The  contraction will follow by  choosing $|\rho|$ and $\lambda$ appropriately so that the product of the operator norms is small.
\begin{theorem}\label{th:remainder0}\sl
Consider $\rho \in \C^n$ such that  $\rho\cdot\rho=0$ and $\lambda>1$. Then
there is a $C_0> 1$, depending on $\|\nabla\log\sigma\|_{\infty}$ and the radius $R$ of $B$, such that
\begin{equation}\label{contra0}
\| {\rm T}^B_{\!\rho} \circ {\rm M}_q \|_{\mathcal{L} \big( X^{1/2}_{\lambda, \rho} ({B}) \big)}   \le 1/2
\end{equation}
whenever $|\rho| > \lambda e^{\lambda R^2}$ and $\lambda= 36C^2_0$.
For all $f \in Y^{-1/2}_{\lambda, \rho, {\rm c}} ({B})$, there is a $w\in X^{1/2}_{\lambda, \rho} ({B})$ such that
\begin{equation}\label{int30} ({\rm I} - {\rm T}^B_{\!\rho} \circ {\rm M}_q) w ={\rm T}^B_{\!\rho}[f]. \end{equation}
Moreover, there is  a $C>1$, depending only on $R$, such that
\begin{align}\label{bounder0}
\| w \|_{X^{1/2}_{\lambda, \rho} ({B})} \le C\|f\|_{Y^{-1/2}_{\lambda, \rho, {\rm c}}(B)}.
\end{align}
\end{theorem}

\begin{proof}
By combining Corollaries~\ref{cor:boundedness} and \ref{cor:M_V}, we have that ${\rm T}^B_{\!\rho} \circ {\rm M}_q$ is
a bounded operator whenever
$|\rho| \geq 4 \lambda R$ .  Furthermore,
there exists a constant $C_0>1$, such that
\begin{align*}
\| {\rm T}^B_{\!\rho} \circ {\rm M}_q \|_{\mathcal{L} \big( X^{1/2}_{\lambda, \rho} (B) \big)} & \leq C_0 \Big(\frac{1}{\lambda^{1/2}|\rho|}+ \frac{1}{\lambda^{1/2}} + \frac{e^{\lambda R^2/2}}{|\rho|^{1/2}} \Big)   \le \frac 12
\end{align*}
whenever $|\rho|^{1/2} > 6C_0e^{\lambda R^2/2}$ and $ \lambda^{1/2} = 6C_0$. Then, by Neumann series, 
${\rm I} - {\rm T}^B_{\!\rho} \circ {\rm M}_q$ has a bounded inverse;
$$
({\rm I} - {\rm T}^B_{\!\rho} \circ {\rm M}_q)^{-1}=\sum_{k\ge 0} ({\rm T}^B_{\!\rho} \circ {\rm M}_q)^k
$$
on $X^{1/2}_{\lambda, \rho} ({B})$ and so $w = ({\rm I} - {\rm T}^B_{\!\rho} \circ {\rm M}_q)^{-1} {\rm T}^B_{\!\rho}[f]$
satisfies \eqref{int30}. Moreover, 
\begin{align*}\label{before}
\| w \|_{X^{1/2}_{\lambda, \rho} ({B})} &\le \sum_{k\ge 0} \big\|  ({\rm T}^B_{\!\rho} \circ {\rm M}_q)^k{\rm T}^B_{\!\rho} [f] \big\|_{X^{1/2}_{\lambda, \rho} ({B})}
\le
2\|{\rm T}^B_{\!\rho}[f]\|_{X^{1/2}_{\lambda, \rho} ({B})},
\end{align*}
by the triangle inequality,  the contraction \eqref{contra0},  and summing the geometric series.  Then~\eqref{bounder0} follows by a final application of Corollary~\ref{cor:boundedness}.
 \end{proof}
 
 Recall that we can also use the trivial inequality \eqref{id} to extend $f\in C^\infty_{\rm c}(B)\mapsto \Delta^{\!-1}_{\rho} f$ as a bounded linear operator
$$\Delta^{\!-1}_{\rho} : \dot{X}^{-1/2}_{\rho} \to  \dot{X}^{1/2}_{\rho}.$$ 
In the following corollary we clarify that the restriction of this extension to the ball~$B$ and the previous locally defined extension ${\rm T}^B_{\!\rho}$ are  the same. We also record the properties of our CGO solutions that we will need in the remaining sections.

 \begin{corollary}\label{th:remainder} Consider $\rho \in \C^n$ and $\lambda>1$ as in Theorem~\ref{th:remainder0}. %Let $q$ be defined in \eqref{qog}.
 Then 
 \begin{equation}\label{contra}
\| \Delta^{\!-1}_{\rho}  \circ {\rm M}_q \|_{\mathcal{L} \big( X^{1/2}_{\lambda, \rho} ({B}) \big)}   \le 1/2,
\end{equation}
there is a $w\in H^1(B)$ that solves \eqref{int}, and there is a $C>1$, depending on $\|\nabla\log\sigma\|_{\infty}$ and the radius $R$ of $B$, such that
%is a $w\in X^{1/2}_{\lambda, \rho} ({B})$ such that
%\begin{equation}\label{int3} ({\rm I} - \Delta^{\!-1}_{\rho} \circ {\rm M}_q) w =\Delta^{\!-1}_{\rho}[q] \end{equation}
%is a solution $w\in H^1(B)$ to the integral equation \eqref{int}
\begin{align}\label{bounder}
\| w \|_{\dot{X}^{1/2}_{ \rho} ({B})} \le C\|q\|_{\dot{X}^{-1/2}_{\rho}}.
\end{align}
 Moreover $v=e_{\rho}(1+w)\in H^1(B)$ solves the Lippmann--Schwinger-type equation
 \begin{equation}\label{defin}({\rm I}- {\rm S}_q)v=e_{\rho},\quad \text{where}\quad {\rm S}_q:=e_{\rho}\Delta_{\rho}^{\!-1}\circ {\rm M}_q[e_{-\rho}\,\centerdot\,],\end{equation} as elements of $H^1(B)$, and is also a weak solution to the Schr\"odinger equation \eqref{schrod}.
 \end{corollary}
 
 \begin{proof} By Corollary~\ref{cor:boundedness}, the equivalence of norms \eqref{equiv}, and the trivial inequality~\eqref{id},
\begin{align*}
\|{\rm T}^B_{\!\rho}g-\Delta^{\!-1}_\rho g\|_{X^{1/2}_{\lambda, \rho} ({B})}&\le \|{\rm T}^B_{\!\rho}[g-g_j]\|_{X^{1/2}_{\lambda,\rho} ({B})}+\|\Delta^{\!-1}_\rho [g_j-g]\|_{X^{1/2}_{\lambda,\rho} ({B})}\\
&\le C\Big( \|g-g_j\|_{Y^{-1/2}_{\lambda,\rho,{\rm c}} ({B})}+\|g_j-g\|_{\dot{X}_{\rho}^{-1/2}}\Big),\end{align*}
and given that the dual norms are also equivalent, by \eqref{equivY}, we can choose  ${g_j\in C_{\rm c}^\infty(B)}$ such that the right-hand side converges to zero. %Thus we find that $${\rm T}^B_{\!\rho}g=\Delta^{\!-1}_\rho g|_{B}$$ as elements of $X^{1/2}_{\lambda, \rho} ({B})$ whenever $g\in Y^{-1/2}_{\lambda, \rho, {\rm c}} ({B})$. 
Then, combining with Corollary~\ref{cor:M_V}, the contraction \eqref{contra} follows directly from the previous contraction~\eqref{contra0}.

Taking $f={\rm M}_q[1]$ in Theorem~\ref{th:remainder0}, we find  $w\in X^{1/2}_{\lambda, \rho} ({B})$ solving
\begin{equation*}\label{toomuch}
w ={\rm T}^B_{\!\rho}\circ {\rm M}_q[1+w].
\end{equation*}
Again by Corollary~\ref{cor:M_V}, we have  ${\rm M}_q[1+w]\in Y^{-1/2}_{\lambda, \rho, {\rm c}} ({B})$, so that, taking this as the function $g$ above, we can also write
\begin{equation}\label{ret}
w =\Delta^{-1}_{\rho}\circ {\rm M}_q[1+w]
\end{equation}
as elements of $X^{1/2}_{\lambda, \rho} ({B})$. Thus, combining with the norm equivalences \eqref{equivall} and \eqref{equiv}, we find that $w\in H^1({B})$ solves~\eqref{int}. Moreover, the inequality~\eqref{bounder} follows from the previous inequality~\eqref{bounder0}  combined with \eqref{equiv} and the dual version of \eqref{equivY}. 

Finally, writing $v=e_\rho(1+w)$ we can multiply \eqref{ret} by $e_\rho$ to find
$$
v -e_\rho=e_\rho \Delta^{-1}_{\rho} \circ {\rm M}_q[e_{-\rho}v]=:{\rm S}_q[v]
$$
as elements of $H^1({B})$. Then, by integration by parts and Plancherel's identity \eqref{planchar}, cancelling the Fourier multipliers,
\begin{align}\begin{split}
-\int_{\R^n} \nabla {\rm S}_q[v]\cdot\nabla \psi%&=-\int \nabla \Big[e_\rho \Delta^{\!-1}_{\rho}\circ {\rm M}_q[e_{-\rho}v]\Big]\cdot \nabla \psi\\
&=\int_{\R^n} \Delta^{\!-1}_{\rho}\circ {\rm M}_q[e_{-\rho}v] e_{\rho}\Delta[e_{-\rho}e_\rho\psi]\\
&=\int_{\R^n} m_\rho^{-1}\widehat{{\rm M}_q[e_{-\rho}v]} m_\rho (e_\rho\psi)^\vee=\langle qv,\psi\rangle\end{split}\label{slap}
\end{align}
whenever $\psi\in C_{\rm c}^\infty(B)$. Given that  $e_\rho$ is harmonic, we see that $v\in H^1({B})$ is also a weak solution to the Schr\"odinger equation \eqref{schrod}.
\end{proof}

\begin{remark}\label{9.2} The CGO solutions $v=e_\rho(1+w)$ given by Corollary~\ref{th:remainder} also satisfy
\begin{equation}\label{sum}
v=
%e_\rho\Big(1+\sum_{k\ge 1} (\Delta^{-1}_{\rho}\circ {\rm M}_q)^k[1]\Big)= 
\sum_{k\ge 0} e_\rho(\Delta^{-1}_{\rho}\circ {\rm M}_q)^k[1],
\end{equation}
with convergence in $H^1(B)$. %By Remark~\ref{tried} and  \eqref{rel} we have that
%\begin{equation}\label{rel2}
%\Delta^{\!-1}\circ {\rm M}_q[e_{\rho}f]=e_{\rho}\Delta^{-1}_{\rho}\circ {\rm M}_q[f]
%\end{equation} 
%as elements of $H^1(B)$ whenever $f\in H^1(B)$, 
 On the other hand, we have that
\begin{align*}\label{going}
e_\rho(\Delta^{-1}_{\rho}\circ {\rm M}_q)^k[1]
%=e_\rho(\Delta^{-1}_{\rho}\circ {\rm M}_q)^k[1]
=(e_{\rho}\Delta^{-1}_{\rho}\circ {\rm M}_{q}[e_{-\rho}\,\centerdot\,])^k[e_\rho]={\rm S}_q^k[e_\rho],
\end{align*}
as elements of $H^1(B)$.  Substituting this into \eqref{sum}, we find that
\begin{equation*}\label{ifonly}
v=\sum_{k\ge 0} {\rm S}_q^k[e_\rho]
\end{equation*}
again in the $H^1(B)$-sense. If we had proven that ${\rm S}_q$ is contractive on $H^1(B)$, we could have solved
 \eqref{defin} more directly by Neumann series and the solution would have taken this form. 
\end{remark}

%%%%%%%%%%%%%%%%%%%% New section %%%%%%%%%%%%%%%%%%%%
\section{The boundary integral identities}\label{boundary}

Here we use  the divergence theorem to equate the boundary integral to an integral over the domain.  Identities similar to the first identity of Lemma~\ref{ale}, often known as Alessandrini identities,  are foundational for the Calder\'on problem. Recall that our boundary integral $BI_{\Lambda_\sigma}:H^{1/2}(\partial \Omega)\times H^{1}(\Omega) \to \mathbb{C}$ is defined by
\begin{equation}\label{dn22} 
BI_{\Lambda_\sigma}(\phi,\psi):=\int_{\partial\Omega} \Big(\sigma^{-1/2}\Lambda_\sigma[\sigma^{-1/2}\phi]-\nu\cdot \nabla P_0[\phi]\Big) \psi
\end{equation}
where $P_0[\phi]$ denotes the harmonic extension of $\phi$. 
% Conveniently there is one less term in the integrand than usual owing to the fact that  our $q$ is already in weak formulation. 
A key idea of Nachman~\cite{zbMATH04105476} and Novikov~\cite{zbMATH04129351} was to take the Faddeev fundamental solution within  boundary integrals similar to this, yielding similar formulas to the second identity in Lemma~\ref{ale}.
 
 \begin{lemma}\label{ale} Let $q$ be defined by \eqref{qog} and let $BI_{\Lambda_\sigma}$ be defined by \eqref{dn22}.  Then
 \begin{equation*}
BI_{\Lambda_\sigma}\big(v|_{\partial\Omega},\psi\big)=\big\langle q v ,\psi\big\rangle
\end{equation*}
whenever $\psi$ is harmonic on $\Omega$ and $v\in H^1(\Omega)$ solves the Schr\"odinger equation \eqref{schrod}. Moreover, 
\begin{equation*}
BI_{\Lambda_\sigma}\big(v|_{\partial\Omega},G_{\!\rho}(x,\,\centerdot\,)\big)={\rm S}_q[v](x)
\end{equation*}
as elements of $H^1(B\setminus \overline{\Omega})$, where ${\rm S}_q$ is defined in \eqref{defin}.
 \end{lemma}
 
 \begin{proof}
For the first identity, consider the weak solution to the conductivity equation given by $u=\sigma^{-1/2}v$. Recalling that $\nabla \sigma$ is bounded almost everywhere, by an application of the product rule, we find that
$
\Delta u=-\sigma^{-1}\nabla\sigma \cdot \nabla u\in L^2(\Omega). 
$
Thus the normal traces can be defined so that the divergence theorem can be applied to $\sigma \nabla u \sigma^{-1/2}\psi -\nabla P_0[\phi]\psi$, yielding
%by the divergence theorem applied  to $\sigma \nabla u \sigma^{-1/2}\psi -\nabla P_0[v|_{\partial \Omega}]\psi$, 
\begin{equation}\label{BIu}
BI_{\Lambda_\sigma}\big(v|_{\partial \Omega},\psi\big)=\int_{\Omega}\Big( \sigma \nabla u \cdot \nabla (\sigma^{-1/2}\psi)- \nabla P_0[v|_{\partial \Omega}] \cdot \nabla \psi\Big);
\end{equation} 
see for example \cite[Proposition 2.4]{zbMATH07578602}.
Now  as $\psi$ is harmonic on $\Omega$, we have
\begin{equation*}
\int_{\Omega}    \nabla(P_0[v|_{\partial \Omega}]-\sigma^{1/2}u)\cdot \nabla\psi=0,
\end{equation*}
 which can be substituted in~\eqref{BIu} to find that 
\begin{align*}
BI_{\Lambda_\sigma}\big(v|_{\partial \Omega},\psi\big) 
    &= \int_{\Omega} \Big(\sigma \nabla u \cdot \nabla (\sigma^{-1/2} \psi)   - \nabla (\sigma^{1/2}u) \cdot \nabla \psi\Big).
\end{align*}
Then, after applying the product rule again, terms cancel and one finds that the right-hand side of this identity is equal to $\big\langle q \sigma^{1/2}u ,\psi\big\rangle=\big\langle q v ,\psi\big\rangle$, as desired.

 For the second identity,  recall that $G_{-\rho}:=e_{-\rho}F_{\!-\rho}$ is a fundamental solution for the Laplacian. In particular  $\Delta G_{-\rho}(\,\centerdot\,,x)=0$ on $\Omega$ for all $x\in B\setminus \overline{\Omega}$. On the other hand, $G_{\!\rho}$ inherits a skew symmetry from $F_{\!\rho}$,
\begin{align}\label{sym}
G_{-\rho}(y,x):\!&=e_{-\rho}(y-x)F_{\!-\rho}(y-x)\\
&=e_{\rho}(x-y)F_{\!\rho}(x-y)=:G_{\!\rho}(x,y),\nonumber
\end{align}
so we can reinterprete this as $\Delta G_{\!\rho}(x,\,\centerdot\,)=0$ on $\Omega$ for all $x\in B\setminus \overline{\Omega}$. Thus, we can substitute this into the first identity to find that
\begin{align}\label{ale2}
BI_{\Lambda_\sigma}\big(v|_{\partial\Omega},G_{\!\rho}(x,\,\centerdot\,)\big)=\big\langle qv, G_{\!\rho}(x,\,\centerdot\,)\big\rangle
\end{align}
for all $x\in B\setminus \overline{\Omega}$. 

Now, for any $f\in H^1(\Omega)$ and any smooth $\psi_x$, supported in a small ball centred at $x$ and properly contained in $B\setminus \overline{\Omega}$, we have that
\begin{align}\label{111}
\int_B \big\langle q f,G_{\!\rho}(y,\,\centerdot\,)\big\rangle\,\psi_x(y)\,\dd y&=\big\langle q f, \!\int_B \! G_{\!\rho}(y,\,\centerdot\,)\,\psi_x(y)\,\dd y\big\rangle.
%&=\big\langle q P_q[\phi], \!\int \! e_{\rho}(x-\,\centerdot\,)F_{\!\rho}(x-\,\centerdot\,)\,\psi(x)\,\dd x\big\rangle\\
%&=\big\langle q P_q[\phi], \!\int \! e_{-\rho}(\,\centerdot\,-x)F_{\!-\rho}(\,\centerdot\,-x)\,\psi(x)\,\dd x\big\rangle\\
%&=\big\langle e_{-\rho} q f, \Delta_{-\rho}^{\!-1}[e_{\rho}\psi]\big\rangle
\end{align}
This follows by interchanging the integral and the gradient, using Lebesgue's dominated convergence theorem, and applying Fubini's theorem. Then using  the skew symmetry~\eqref{sym} again, and the kernel representation \eqref{loveit} of $\Delta^{\!-1}_\rho$, the right-hand side of \eqref{111} can be rewritten as
\begin{align}\label{222}
\big\langle  q f, e_{-\rho}\Delta_{-\rho}^{\!-1}[e_{\rho}\psi_x]\big\rangle&=\int_B e_{\rho}\Delta_{\rho}^{\!-1}\big[q fe_{-\rho} \big](y) \,\psi_x(y)\,\dd y.
%&=\int \Delta^{\!-1}\circ {\rm M}_q\circ P_q[\phi](x)\,\psi(x)\,\dd x.
\end{align}
Here we have considered $\Delta_\rho^{\!-1}$ to be the globally defined extension given by \eqref{id} and the identity follows by moving the Fourier multiplier $m_{\rho}^{-1}$ from one term to the other after an application of Plancherel's identity \eqref{planchar}.
Combining \eqref{111} with~\eqref{222}, and recalling the definition \eqref{defin} of ${\rm S}_q$ we find that
$$
\int_B \big\langle q f,G_{\!\rho}(y,\,\centerdot\,)\big\rangle\,\psi_x(y)\,\dd y=\int_B {\rm S}_q[f](y) \,\psi_x(y)\,\dd y.
$$
Now by the bounds of the previous section, we have that ${\rm S}_q[f]\in H^1(B)$, and so letting~$\psi_x$ approximate the Dirac delta $\delta_x$, we find that
\begin{equation}\label{finalll}
\big\langle q f, G_{\!\rho}(x,\,\centerdot\,)\big\rangle={\rm S}_q[f](x)
\end{equation}
%for all $y\in B\setminus \overline{\Omega}$. Now as explained in Remark~\ref{tried}, $\Delta_\rho^{\!-1}[e_{-\rho}qf](y)={\rm T}^B_{\!\rho}[e_{-\rho}qf](y)$ for almost every $y\in B$ whenever $f\in H^1(\Omega)$, so this pointwise relation continues to hold after replacing $\Delta_\rho^{\!-1}$ with ${\rm T}^B_{\!\rho}$, so that
%\begin{equation}\label{finalll}
%\big\langle q f, G_{\!\rho}(y,\,\centerdot\,)\big\rangle=\Delta^{\!-1}\circ {\rm M}_q[f](y)
%\end{equation}
for almost every $x\in B\setminus \overline{\Omega}$ by a suitable version of the Lebesgue differentiation theorem; see for example \cite[Theorem 2.12]{zbMATH06136466}. Taking $f=v$ and combining \eqref{ale2} with \eqref{finalll} yields the second identity.
 \end{proof}
 
 %%%%%%%%%%%%%%%%%%%% New section %%%%%%%%%%%%%%%%%%%%
\section{The proofs of Theorems~\ref{bound} and \ref{final}}\label{finalsection}

 The second identity of Lemma~\ref{ale} allows us to define $\Gamma_{\!\!\Lambda_\sigma}:H^{1/2}(\partial \Omega)\to H^{1/2}(\partial \Omega)$ by taking the outer trace ${\rm T}_{\partial \Omega}: H^{1}(B\setminus \overline{\Omega})\to H^{1/2}(\partial \Omega)$ of  the boundary integral;
 \begin{equation}\label{newgam}
 \Gamma_{\!\!\Lambda_\sigma}[\phi]:={\rm T}_{\partial \Omega}\big[BI_{\Lambda_\sigma}\big(\phi,G_{\!\rho}({\rm x},\,\centerdot\,)\big)\big]
\end{equation}
for all $\phi\in H^{1/2}(\partial \Omega)$. Moreover it gives us the alternative representation
 \begin{equation}\label{gammalt}
\Gamma_{\!\!\Lambda_\sigma}[\phi]={\rm T}_{\partial \Omega} \circ {\rm S}_q\circ P_q[\phi],
\end{equation}
where ${\rm S}_q$ is defined in \eqref{defin} and $P_q[\phi]$ denotes the solution to \eqref{schrod} with Dirichlet data~$\phi$.

 We restate the main theorems from Section~\ref{recon} before proving them.
 The proof of the second part of the  following theorem bears some resemblance to the argument of \cite[Theorem 3.1]{zbMATH06680012}, allowing us to avoid the use of double layer potentials.

\begin{main-lemma}\sl  Consider $\rho \in \C^n$ such that  $\rho\cdot\rho=0$ with $|\rho|$ as large as in Theorem~\ref{th:remainder0}. Let  $\Gamma_{\!\!\Lambda_\sigma}$ be defined in \eqref{newgam}.
 Then 
\vspace{0.3em}\begin{itemize}
\item [{(i)}] $\Gamma_{\!\!\Lambda_\sigma}:H^{1/2}(\partial \Omega)\to H^{1/2}(\partial \Omega)$ is bounded compactly,
\item [{(ii)}] if $\Gamma_{\!\!\Lambda_\sigma}[\phi]=\phi$ then $ \phi=0$,
\item [{(iii)}] $\mathrm{I}-\Gamma_{\!\!\Lambda_\sigma}$ has a bounded inverse on $H^{1/2}(\partial \Omega)$,
\end{itemize}
\vspace{0.2em}
and if $v\in H^1(B)$ solves the Lippmann--Schwinger-type equation \eqref{defin},  then
\vspace{0.2em}\begin{itemize}
\item [{(iv)}]  
 $v|_{\partial \Omega}= (\mathrm{I}-\Gamma_{\!\!\Lambda_\sigma})^{-1}[e_{\rho}|_{\partial \Omega}].
$
\end{itemize}
\end{main-lemma}

\begin{proof}  By hypothesis $({\rm I}- {\rm S}_q)v=e_{\rho}$, so part {\sl(iv)} follows from the alternative representation~\eqref{gammalt} and part {\sl(iii)}, which in turn will follow from parts {\sl(i)} and~{\sl(ii)} by the Fredholm alternative.

To see {\sl(i)}, note first that  the trace operator ${\rm T}_{\partial \Omega}: H^{1}(B\setminus \overline{\Omega})\to H^{1/2}(\partial \Omega)$ and solution operator $P_q:H^{1/2}(\partial \Omega) \to H^{1}(\Omega)$ are  bounded. Combining this with the alternative representation \eqref{gammalt},  it will suffice to show that ${\rm S}_q:H^1(\Omega)\to H^{1}(B\setminus\overline{\Omega})$ is bounded compactly. 
For this we  recall that on $B\setminus \overline{\Omega}$ we have the representation~\eqref{finalll}, and so by applications of the product rule we can divide the operator into three parts ${\rm S}_q={\rm S}_1+{\rm S}_2+{\rm S}_3$, where
$$
{\rm S}_1[f]:=\frac{1}{4} \int_{\Omega} |\nabla \log \sigma(y)|^2 f(y)G_{\!\rho}(\,\centerdot\,-y)\,\dd y,
$$
$$
 {\rm S}_2[f]:=-\frac{1}{2} \int_{\Omega} \nabla \log \sigma(y) \cdot \nabla f(y)G_{\!\rho}(\,\centerdot\,-y)\, \dd y,
 $$
 and
 $$
  {\rm S}_3[f]:=\frac{1}{2} \int_{\Omega} \nabla \log \sigma(y) \cdot \nabla G_{\!\rho}(\,\centerdot\,-y) f(y)\,\dd y.
$$
By our {\it a priori} assumptions $\nabla \log \sigma=\sigma^{-1}\nabla \sigma\in L^\infty(\Omega)^n$ and on the other hand  $G_{\!\rho}$ and $\nabla G_{\!\rho}$ are locally integrable by \eqref{bigpot}. Thus by Young's convolution inequality, $${\rm S}_1: L^2(\Omega)\to L^2(B\setminus\overline{\Omega}),\quad {\rm S}_2 : H^{1}(\Omega) \to L^2(B\setminus\overline{\Omega}),\quad  \text{and}\quad {\rm S}_3 : L^2(\Omega)\to L^2(B\setminus\overline{\Omega})$$ are bounded.  Moreover, by Lebesgue's dominated convergence theorem, we can take derivatives under the integral, and by \eqref{bigpot} we have that
$$
\partial_{{\rm x}_j} \partial_{{\rm x}_i}G_{\!\rho}({\rm x}-{\rm y})=  c_nn\frac{({\rm x}_j-{\rm y}_j)({\rm x}_i-{\rm y}_i)}{|{\rm x}-{\rm y}|^{n+2}} + \partial_{{\rm x}_j} \partial_{{\rm x}_i}H_{\!\rho}({\rm x} - {\rm y}).
$$
On the one hand,  the second order Riesz transforms are easily bounded in $L^2$ noting that the Fourier multipliers $-{\rm \xi}_j{\rm \xi}_i/|{\rm \xi}|^2$ are uniformly bounded; see for example \cite[Section 7.2]{zbMATH06136466}. On the other hand, the operator corresponding to the second term can be bounded in $L^2(B\setminus\overline{\Omega})$ by Young's inequality again. Together we find that
 $${\rm S}_1: L^2(\Omega)\to H^2(B\setminus\overline{\Omega}), \quad  {\rm S}_2 : H^{1}(\Omega) \to H^2(B\setminus\overline{\Omega})\quad \text{and}\quad {\rm S}_3 : L^2(\Omega)\to H^{1}(B\setminus\overline{\Omega})$$ are bounded. Thus, by Rellich's theorem, all three operators are bounded from $H^{1}(\Omega) \to H^{1}(B\setminus\overline{\Omega})$ compactly. Altogether we find that ${\rm S}_q:H^{1}(\Omega) \to H^{1}(B\setminus\overline{\Omega})$ compactly, which completes the proof of {\sl(i)}.

In order to see {\sl(ii)}, we combine its hypothesis with the alternative representation~\eqref{gammalt}, obtaining
\begin{align} \label{trace}
\phi =\Gamma_{\!\!\Lambda_\sigma}[\phi] ={\rm T}_{\partial \Omega} \circ {\rm S}_q\circ P_q[\phi].
\end{align}
By the bounds of Section~\ref{contraction}, we know that ${\rm S}_q\circ P_q[\phi]=e_{\rho}\Delta^{\!-1}_{\rho}\circ M_{q}[e_{-\rho} P_q[\phi]] \in H^1(B),$ so we can replace the outer trace on $\partial\Omega$ with the inner trace as they both extend the restriction to~$\partial\Omega$ of smooth functions, which are dense in $H^1(B)$.
On the other hand, combining the calculation \eqref{slap} with the defining property of $P_q[\phi]$, we have that
\begin{align*}
-\int_{\R^n} \nabla ({\rm S}_q\circ P_q[\phi])\cdot\nabla \psi
%&=-\int \nabla \Big[e_\rho \Delta^{\!-1}_{\rho}\circ {\rm M}_q\big[e_{-\rho}P_q[\phi]\big]\Big]\cdot \nabla \psi\\
%&=\int \Delta^{\!-1}_{\rho}\circ {\rm M}_q\big[e_{-\rho}P_q[\phi]\big] e_{\rho}\Delta[e_{-\rho}e_\rho\psi]\\
%&=\int m_\rho^{-1}({\rm M}_q\big[e_{-\rho}P_q[\phi]\big])^\wedge \,m_\rho (e_\rho\psi)^\vee\\
&=\langle qP_q[\phi],\psi\rangle=-\int_{\R^n} \nabla P_q[\phi]\cdot\nabla \psi
\end{align*}
whenever $\psi\in C_{\rm c}^\infty(\Omega)$, and so 
$\Delta[{\rm S}_q\circ P_q[\phi]-P_q[\phi]]=0$ in $\Omega$
in the weak sense. Combining this with our hypothesis \eqref{trace} and  the uniqueness of solutions for the Dirichlet problem with zero boundary data, we find that 
\begin{align}\label{eq}
{\rm S}_q\circ P_q[\phi] = P_q[\phi] {\quad}  \text{in}\ \Omega.
\end{align}
If we had a contraction for ${\rm S}_q$, it would be easier to conclude that $\phi=0$. In any case, we can use the contraction we have by considering \begin{align*}
\eta&:= e_{-\rho}{\rm S}_q \circ P_q[\phi]=\Delta^{\!-1}_{\rho}\circ {\rm M}_q \big[e_{-\rho} P_q[\phi]\big]=\Delta^{\!-1}_{\rho} \circ {\rm M}_q[\eta],
\end{align*}
 where the final identity follows from the definition of $\eta$ and \eqref{eq}. 
Then our contraction \eqref{contra} implies that  $\eta$ must be the zero element of $X^{1/2}_{\lambda, \rho} ({B})$, so by the equivalence of the norms $e_{\rho}\eta$ must be the zero element of $H^1(B)$.  Then by the definition of $\eta$ and \eqref{eq} again, $P_q[\phi]$ is the zero element of $H^1(\Omega)$. Finally, by uniqueness of the Dirichlet problem,  $\phi$ is the zero element of $H^{1/2}(\partial\Omega)$, which completes the proof of the injectivity.
\end{proof}

 \begin{remark}\label{street} Much of the previous argument is insensitive to the choice of fundamental solutions used to invert $\Delta$ and $\Delta_\rho$.
Rather than troubling ourselves to invert $\Delta_\rho$ using the Faddeev fundamental solution, we could have more easily inverted the operator using the {\it a priori} estimates proved in the uniqueness result of~\cite{zbMATH06534426}. 
Indeed, 
we were able to use those estimates to find a different fundamental solution~$K_\rho$ and~$w$ so that
\[w ({\rm x})-\big\langle qw, K_\rho ({\rm x},\,\centerdot\,)\big\rangle = \big\langle q, K_\rho ({\rm x},\,\centerdot\,)\big\rangle {\quad}  \text{in}\ B\setminus\overline{\Omega}.\]
% For this we proved that the operator with kernel $K_\rho$ composed with ${\rm M}_q$ is a contraction in spaces somewhat similar to those described earlier.
  The associated CGO solutions $v=e_{\rho}(1+w)$ satisfy 
\begin{equation*}\label{lip} v({\rm x})-\big\langle q v, L_\rho({\rm x},\centerdot\,)\big\rangle=e_{\rho}({\rm x}){\quad}  \text{in}\ B\setminus\overline{\Omega},
\end{equation*}
where $L_\rho({\rm x},{\rm y}):=e_{\rho}({\rm x}-{\rm y})K_\rho ({\rm x},{\rm y})$ as before.
However, not only are these fundamental solutions less explicitly defined, they also 
fail to satisfy the skew symmetry law \eqref{sym}, that $K_{-\rho} ({\rm x},{\rm y})=K_\rho ({\rm y},{\rm x})$. Thus, even though we know that $L_\rho(\,\centerdot\,,y)$ is harmonic on $\mathbb{R}^n\setminus \{y\}$, one is unable to conclude that $L_\rho(x,\centerdot\,)$ is harmonic on $\mathbb{R}^n\setminus \{x\}$,  which is what allowed us to take it in the boundary integral identity. We attempted to modify the fundamental solution so that
the skew symmetry law is satisfied as in~\cite{zbMATH05690567}, however we were unable to do this while maintaining the contraction.
\end{remark}

We are now ready to complete the formula for the Fourier transform $
\widehat{q}({\rm \xi}):=\langle q, e^{-i {\rm \xi}\cdot {\rm x}}\rangle
$ with $q$ defined in \eqref{qog}. The proof  makes use of the boundary integral identity again combined with the averaging argument due to Haberman and Tataru~\cite{zbMATH06145493}.

\begin{main-theorem} \sl Let $ \Pi$ be a two-dimensional linear subspace orthogonal to $ \xi \in \R^n$ and define
$$
S^1:= \Pi \cap \big\{\,\theta \in \R^n : |\theta| = 1\, \big\}.$$
 For  $ \theta \in S^1$,
 let $ \vartheta \in S^1$ be such that $ \{\theta,\vartheta \} $ is an orthonormal basis of~$ \Pi $ and define
 \begin{align*}
\rho := \tau \theta + i\Big(-\frac{\xi}{2} + \Big( \tau^2 - \frac{|\xi|^2}{4} \Big)^{1/2} \vartheta \Big),\quad \rho' := -\tau \theta + i\Big(-\frac{\xi}{2} - \Big( \tau^2 - \frac{|\xi|^2}{4} \Big)^{1/2} \vartheta \Big),
\end{align*}
where $\tau>1$. Let $BI_{\Lambda_\sigma}$ and $\Gamma_{\!\!\Lambda_\sigma}$ be defined in \eqref{dn22} and \eqref{newgam}, respectively. Then
\begin{align*}
\widehat{q}(\xi)=
 \lim_{T\to \infty}\frac{1}{2\pi T} \int_{T}^{2T}\!\! \int_{S^1} BI_{\Lambda_\sigma}\Big((\mathrm{I}-\Gamma_{\!\!\Lambda_\sigma})^{-1}[e_{\rho}|_{\partial \Omega}], e_{\rho'}\Big)   \,  \dd \theta \dd \tau.
\end{align*}
\end{main-theorem}

\begin{proof} 
Noting that $\rho\cdot\rho=\rho'\cdot\rho'=0$, we can take the CGO solution $v=e_{\rho}(1+w)\in H^1(B)$ given by Corollary~\ref{th:remainder} and $\psi=e_{\rho'}$ in the first boundary integral identity of Lemma~\ref{ale}. Noting also that $\rho+\rho'=-i\xi$,  the right-hand side of the identity can be written as $\widehat{q}(\xi)$ plus a remainder term. Indeed we find that
\begin{equation}\label{or}
BI_{\Lambda_\sigma}\big(v|_{\partial \Omega},e_{\rho'}\big)=\widehat{q}(\xi)+\big\langle qw,e^{-i \xi\cdot {\rm x}}\big\rangle.
\end{equation}
Now, for any extension $\widetilde{w}\in \dot{X}^{1/2}_\rho$ of $w$, and smooth $\chi_B$,  equal to one on $\Omega$ and supported on $B$, by duality we have that
\begin{align*}
\big|\big\langle qw,e^{-i \xi\cdot {\rm x}}\big\rangle\big|& \le \|q\|_{\dot{X}_{\rho}^{-1/2}}\|\chi_B e^{-i \xi\cdot {\rm x}}\widetilde{w}\|_{\dot{X}_{\rho}^{1/2}}\nonumber\\
&\le C\|q\|_{\dot{X}_{\rho}^{-1/2}}\|\widetilde{w}\|_{\dot{X}_{\rho}^{1/2}}\nonumber
\end{align*}
where the constant $C>1$ depends on $|\xi|$ and $R$; see \cite[Lemma 2.2]{zbMATH06145493} or \cite[(3.17)]{zbMATH06117512}. Taking the infimum over  extensions we find 
\begin{align*}
\big|\big\langle qw,e^{-i \xi\cdot {\rm x}}\big\rangle\big|\le C\|q\|_{\dot{X}_{\rho}^{-1/2}}\|w\|_{\dot{X}_{\rho}^{1/2}(B)}.\nonumber
\end{align*}
Then using the estimate \eqref{bounder} for the remainder in Corollary~\ref{th:remainder}, and taking an average over~$\rho$, we find that
$$
\frac{1}{2\pi T}  \int_{T}^{2T}\!\! \int_{S^1}\big|\big\langle qw,e^{-i \xi\cdot {\rm x}}\big\rangle\big|\,  \dd \theta \dd \tau\le \frac{C}{2\pi T} \int_{T}^{2T}\!\! \int_{S^1}\|q\|^2_{\dot{X}_{\rho}^{-1/2}}\,  \dd \theta \dd \tau,
$$
where $C>1$ depends on $|\xi|$, the radius $R$, and $\|\nabla\log\sigma\|_{\infty}$.
Now, Haberman and Tataru \cite[Lemma 3.1]{zbMATH06145493} proved that the right-hand side converges to zero as $T\to\infty$.  Combining with \eqref{or}, noting that $\widehat{q}(\xi)$ is unchanged by the average, yields 
$$
\widehat{q}(\xi)= \lim_{T\to \infty}\frac{1}{2\pi T} \int_{T}^{2T}\!\! \int_{S^1} BI_{\Lambda_\sigma}\big(v|_{\partial \Omega},e_{\rho'}\big)\,  \dd \theta \dd \tau.
$$
Finally, we can use our formula for the values of $v$ on the boundary given by Theorem~\ref{bound}, which completes the proof.
\end{proof}

\begin{remark}\label{dif}
In \cite{zbMATH06490961, zbMATH07395052, zbMATH07373390}, the contraction was found after taking similar averages over~$\rho$, which yields the existence of a sequence of CGO solutions $$\big\{v_j=e_{\rho_j}(1+w_j) \big\}_{j\ge1}\quad \text{with}\quad |\rho_j|\to \infty\quad \text{as}\quad j\to \infty.$$ The authors of \cite{zbMATH06490961, zbMATH07395052, zbMATH07373390} were able to take advantage of the existence of these solutions to prove uniqueness, however in order to reconstruct in terms of these solutions, one would need to know which values of $\rho_j\in\mathbb{C}^n$ to take. 
\end{remark}

%%%%%%%%%%%%%%%%%%%New section%%%%%%%%%%%%%%%%%
\section{Reconstruction in practise}\label{pract}

There is an extensive literature dedicated to the real-world practicalities of the Calder\'on problem, such as stability, partial data and numerical implementation; see for example \cite{zbMATH06517077, zbMATH06033828,
zbMATH05180743}. Here we suggest some simplifications that would make things easier to measure and calculate without dwelling on how much the simplifications would corrupt the image.

\subsection{What to measure:} An approximation of the conductivity on the surface $\sigma|_{\partial\Omega}$ could be measured directly by placing real potential differences over pairs of adjacent electrodes, measuring the induced current, and applying Ohm's law. %If the distances between each pair were small this could provide a relatively good approximation. 
Earlier reconstruction algorithms also required the perpendicular gradient of the conductivity on the surface, which seems harder to measure directly.   We would also need to measure an approximation of 
$$
{\rm Meas}_T(\xi):=\frac{1}{2\pi T}\int_{T}^{2T}\!\! \int_{S^1} \int_{\partial\Omega}\Lambda_\sigma[\sigma^{-1/2}e_{\rho}]\sigma^{-1/2}e_{\rho'}  \,\dd \theta \dd \tau  
$$ 
for all $\xi\in R^{-1}\mathbb{Z}^n\cap [-cT,cT]^n$, where $cT>1$ and $R$ is approximately twice the diameter of $\Omega$. For the complex integrand one can place two separate real electric potentials.
Given sufficient access to a large enough part of the surface, one would hope to approximate the inner integral with some accuracy, however applying the oscillating electric potentials could prove to be the more difficult technical challenge. 
The outer averaged integrals seem less important and a more rudimentary finite sum approximation could be sufficient. 

\subsection{What to calculate:} Given ${\rm Meas}_T$ and  $\sigma|_{\partial\Omega}$,  one could then employ a triangular finite element method to calculate an approximate solution to \begin{equation*}\label{cond2}
	\left\{
		\begin{aligned}
		\Delta v &=(\Re q_T) v&  \text{in}&\ \Omega , \\
		v&= \sigma|^{1/2}_{\partial\Omega}&  \text{on}&\ \partial\Omega ,
		\end{aligned}
	\right.
\end{equation*}
where, letting $\mathbf{1}_{\Omega}$ denote the characteristic function of the domain,  $q_T$ is defined by
$$
q_{T}({\rm{x}}):=\frac{1}{(2\pi R)^n}\sum_{\xi\in R^{-1}\mathbb{Z}^n\cap [-cT,cT]^n} e^{i{\rm x}\cdot\xi} \Big({\rm Meas}_T(\xi)+\tfrac{|\xi|^2}{2}\widehat{\mathbf{1}_{\Omega}}(\xi)\Big).
$$
Then the grayscale image is given by $v^2$, taking $T$ as large as is practicable. 

\subsection{Justification of the simplifications: } A loose interpretation of Theorem~\ref{bound} is that $v|_{\partial\Omega}$ is not so different from $e_\rho|_{\partial\Omega}$. Indeed if the conductivity were constant, then~$\Gamma_{\!\!\Lambda_\sigma}$ would be identically zero and so part of the reconstruction integral from Theorem~\ref{final} could be rewritten using the divergence theorem;
\begin{equation*}\label{bit}
\int_{\partial\Omega} \partial_{\nu} P_0[e_{\rho}] e_{\rho'} =\int_{\Omega} \nabla e_{\rho} \cdot \nabla e_{\rho'}=\rho\cdot\rho'\int_{\Omega} e^{-i\xi\cdot{\rm x}}=-\tfrac{|\xi|^2}{2}\widehat{\mathbf{1}_{\Omega}}(\xi).
\end{equation*}
Note also that, by the uncertainty principle, $\widehat{q}$ and $\widehat{\mathbf{1}_\Omega}$ are essentially constant at scale~$R^{-1}$. Thus the reconstruction formula approximately simplifies to 
$
\widehat{q}\approx \lim_{T\to \infty}\widehat{q}_{T}$ pointwise.
Note that the cutoff of the frequencies  serves to mollify so that $q_T$ is a function even though it approximately converges to $q$ in the distributional sense. Finally, one observes that $\sigma^{1/2}$ is the unique solution to the Schr\"odinger equation with $v|_{\partial\Omega}=\sigma|^{1/2}_{\partial\Omega}$. %This is the approach taken by Nachman \cite[Theorem 5.1]{zbMATH04105476} to extract $\sigma$ from $q$.

%%%%%%%%%%%%%%%%%%%% New section %%%%%%%%%%%%%%%%%%%%


\begin{thebibliography}{10}

\bibitem{zbMATH06680012}
Kari Astala, Daniel Faraco, and Keith~M. Rogers.
\newblock Unbounded potential recovery in the plane.
\newblock {\em Ann. Sci. {\'E}c. Norm. Sup{\'e}r. (4)}, 49(5):1027--1051, 2016.

\bibitem{zbMATH05077070}
Kari {Astala} and Lassi {P\"aiv\"arinta}.
\newblock A boundary integral equation for {Calder{\'o}n}'s inverse
  conductivity problem.
\newblock {\em Collect. Math.}, 2006:127--139, 2006.

\bibitem{zbMATH05050053}
Kari {Astala} and Lassi {P\"aiv\"arinta}.
\newblock {Calder\'on's inverse conductivity problem in the plane.}
\newblock {\em {Ann. Math. (2)}}, 163(1):265--299, 2006.

\bibitem{zbMATH03457577}
Colin Bennett.
\newblock Banach function spaces and interpolation methods. {I}: {The} abstract
  theory.
\newblock {\em J. Funct. Anal.}, 17:409--440, 1974.

\bibitem{zbMATH00912089}
Russell~M. {Brown}.
\newblock {Global uniqueness in the impedance-imaging problem for less regular
  conductivities.}
\newblock {\em {SIAM J. Math. Anal.}}, 27(4):1049--1056, 1996.

\bibitem{zbMATH01731190}
Russell~M. {Brown}.
\newblock {Recovering the conductivity at the boundary from the Dirichlet to
  Neumann map: A pointwise result.}
\newblock {\em {J. Inverse Ill-Posed Probl.}}, 9(6):567--574, 2001.

\bibitem{zbMATH02102106}
Russell~M. {Brown} and Rodolfo~H. {Torres}.
\newblock {Uniqueness in the inverse conductivity problem for conductivities
  with $3/2$ derivatives in $L^p$, $p>2n$.}
\newblock {\em {J. Fourier Anal. Appl.}}, 9(6):563--574, 2003.

\bibitem{zbMATH05684831}
Alberto~P. {Calder\'on}.
\newblock {On an inverse boundary value problem.}
\newblock {\em {Comput. Appl. Math.}}, 25(2-3):133--138, 1980.

\bibitem{zbMATH06517077}
Pedro Caro, David Dos Santos~Ferreira, and Alberto Ruiz.
\newblock Stability estimates for the {Calder{\'o}n} problem with partial data.
\newblock {\em J. Differ. Equations}, 260(3):2457--2489, 2016.

\bibitem{zbMATH06117512}
Pedro {Caro}, Andoni {Garc{\'\i}a}, and Juan~Manuel {Reyes}.
\newblock {Stability of the Calder\'on problem for less regular
  conductivities.}
\newblock {\em {J. Differ. Equations}}, 254(2):469--492, 2013.

\bibitem{zbMATH06534426}
Pedro {Caro} and Keith~M. {Rogers}.
\newblock {Global uniqueness for the Calder\'on problem with Lipschitz
  conductivities.}
\newblock {\em {Forum Math. Pi}}, 4:28, 2016.

\bibitem{zbMATH06033828}
Fabrice Delbary, Per~Christian Hansen, and Kim Knudsen.
\newblock Electrical impedance tomography: 3d reconstructions using scattering
  transforms.
\newblock {\em Appl. Anal.}, 91(4):737--755, 2012.

\bibitem{zbMATH03237162}
Ludvig~D. Faddeev.
\newblock Increasing solutions of the {Schr{\"o}dinger} equation.
\newblock Sov. {Phys}., {Dokl}. 10 (1965), 1033-1035 (1966); translation from
  {Dokl}. {Akad}. {Nauk} {SSSR} 165, 514-517 (1965)., 1965.

\bibitem{zbMATH06659335}
Andoni {Garc\'{\i}a} and Guo {Zhang}.
\newblock {Reconstruction from boundary measurements for less regular
  conductivities}.
\newblock {\em {Inverse Probl.}}, 32(11):22, 2016.
\newblock Id/No 115015.

\bibitem{zbMATH01554166}
David Gilbarg and Neil~S. Trudinger.
\newblock {\em Elliptic partial differential equations of second order}.
\newblock Class. Math. Berlin: Springer, reprint of the 1998 ed. edition, 2001.

\bibitem{zbMATH06490961}
Boaz {Haberman}.
\newblock {Uniqueness in Calder\'on's problem for conductivities with unbounded
  gradient.}
\newblock {\em {Commun. Math. Phys.}}, 340(2):639--659, 2015.

\bibitem{zbMATH06145493}
Boaz {Haberman} and Daniel {Tataru}.
\newblock {Uniqueness in Calder\'on's problem with Lipschitz conductivities.}
\newblock {\em {Duke Math. J.}}, 162(3):497--516, 2013.

\bibitem{zbMATH07373390}
Seheon Ham, Yehyun Kwon, and Sanghyuk Lee.
\newblock Uniqueness in the {Calder{\'o}n} problem and bilinear restriction
  estimates.
\newblock {\em J. Funct. Anal.}, 281(8):58, 2021.
\newblock Id/No 109119.

\bibitem{zbMATH00764042}
David Jerison and Carlos~E. Kenig.
\newblock The inhomogeneous {Dirichlet} problem in {Lipschitz} domains.
\newblock {\em J. Funct. Anal.}, 130(1):161--219, 1995.

\bibitem{zbMATH05180743}
Carlos~E. Kenig, Johannes Sj{\"o}strand, and Gunther Uhlmann.
\newblock The {Calder{\'o}n} problem with partial data.
\newblock {\em Ann. Math. (2)}, 165(2):567--591, 2007.

\bibitem{zbMATH07578602}
Hyunseok Kim and Hyunwoo Kwon.
\newblock Dirichlet and {Neumann} problems for elliptic equations with singular
  drifts on {Lipschitz} domains.
\newblock {\em Trans. Am. Math. Soc.}, 375(9):6537--6574, 2022.

\bibitem{zbMATH03939884}
Robert~V. Kohn and Michael Vogelius.
\newblock Determining conductivity by boundary measurements.
\newblock {\em Commun. Pure Appl. Math.}, 37:289--298, 1984.

\bibitem{zbMATH03957806}
Robert~V. Kohn and Michael Vogelius.
\newblock Determining conductivity by boundary measurements. {II}: {Interior}
  results.
\newblock {\em Commun. Pure Appl. Math.}, 38:643--667, 1985.

\bibitem{zbMATH01286366}
Marius Mitrea and Michael Taylor.
\newblock Boundary layer methods for {Lipschitz} domains in {Riemannian}
  manifolds.
\newblock {\em J. Funct. Anal.}, 163(2):181--251, 1999.

\bibitem{zbMATH06136466}
Camil Muscalu and Wilhelm Schlag.
\newblock {\em Classical and multilinear harmonic analysis. {Volume} {I}},
  volume 137 of {\em Camb. Stud. Adv. Math.}
\newblock Cambridge: Cambridge University Press, 2013.

\bibitem{zbMATH04105476}
Adrian {Nachman}.
\newblock {Reconstructions from boundary measurements.}
\newblock {\em {Ann. Math. (2)}}, 128(3):531--576, 1988.

\bibitem{zbMATH05690567}
Adrian {Nachman} and Brian {Street}.
\newblock {Reconstruction in the Calder\'on problem with partial data}.
\newblock {\em {Commun. Partial Differ. Equations}}, 35(2):375--390, 2010.

\bibitem{zbMATH04050176}
Adrian Nachman, John Sylvester, and Gunther Uhlmann.
\newblock An {{\(n\)}}-dimensional {Borg}-{Levinson} theorem.
\newblock {\em Commun. Math. Phys.}, 115(4):595--605, 1988.

\bibitem{zbMATH00194126}
Roger~G. Newton.
\newblock {\em Inverse {Schr{\"o}dinger} scattering in three dimensions}.
\newblock Texts Monogr. Phys. London etc.: Springer-Verlag, 1989.

\bibitem{zbMATH04129351}
Roman~G. {Novikov}.
\newblock {Multidimensional inverse spectral problem for the equation $-\Delta
  \psi -(v(x)-Eu(x))\psi =0$.}
\newblock {\em {Funct. Anal. Appl.}}, 22(4):263--272, 1988.

\bibitem{zbMATH02005267}
Lassi {P\"aiv\"arinta}, Alexander {Panchenko}, and Gunther {Uhlmann}.
\newblock {Complex geometrical optics solutions for Lipschitz conductivities.}
\newblock {\em {Rev. Mat. Iberoam.}}, 19(1):57--72, 2003.

\bibitem{zbMATH07395052}
Felipe Ponce-Vanegas.
\newblock A bilinear strategy for {Calder{\'o}n}'s problem.
\newblock {\em Rev. Mat. Iberoam.}, 37(6):2119--2160, 2021.

\bibitem{zbMATH04015323}
John {Sylvester} and Gunther {Uhlmann}.
\newblock {A global uniqueness theorem for an inverse boundary value problem.}
\newblock {\em {Ann. Math. (2)}}, 125:153--169, 1987.

\bibitem{UhlmannICM}
Gunther {Uhlmann}.
\newblock {Inverse boundary value problems for partial differential equations.}
\newblock {\em {Doc. Math.}}, Vol. III:77--86, (Berlin, 1998).

\bibitem{zbMATH06012150}
Wei Zhang and Jiguang Bao.
\newblock Regularity of very weak solutions for elliptic equation of divergence
  form.
\newblock {\em J. Funct. Anal.}, 262(4):1867--1878, 2012.

\end{thebibliography}
\end{document}